\documentclass{amsart}
\usepackage{amscd}
\usepackage{hyperref}
\usepackage{xcolor}
\usepackage{amssymb}

\newtheorem{theorem}{Theorem}[section]

\newtheorem{lemma}{Lemma}[section]

\numberwithin{equation}{section}

\newcommand{\esp}{\mathbf{L}}
\newcommand{\F}{\mathcal F}
\newcommand{\B}{\mathcal{B}}

\begin{document}
\title[Boundary value problem for 1D NLS]{Inhomogeneous Mixed-boundary
value problem for one dimensional nonlinear Schr\"{o}dinger equations via
factorization techniques}
\author[L. Esquivel]{L. Esquivel}
\address{Gran Sannso Science Institute\\ L'Aquila, Italia}
\email{liliane.esquivel@gssi.it}
\author[N. Hayashi]{Nakao Hayashi}
\address{Department of Mathematics\\
Graduate School of Science, Osaka University, Osaka, Toyonaka 560-0043, Japan%
}
\email{nhayashi@math.sci.osaka-u.ac.jp}
\author[E. Kaikina]{Elena I. Kaikina}
\address{Centro de Ciencias Matem\'{a}ticas\\
UNAM Campus Morelia, AP 61-3 (Xangari), Morelia CP 58089, Michoac\'{a}n,
Mexico}
\email{ekaikina@matmor.unam.mx}

\begin{abstract}
We consider the inhomogeneous Mixed-boundary value problem for the cubic
nonlinear Schr\"{o}dinger equations on the half line. We present sufficient
conditions of initial and boundary data which ensure asymptotic behavior of
small solutions to equations by using the classical energy method and
factorization techniques
\end{abstract}

\maketitle
\section{Introduction\label{Section 0}}

We consider the initial-boundary value problem for the nonlinear Schr\"{o}%
dinger equations on the half-line 
\begin{equation}
\left\{ 
\begin{array}{c}
Lu=f(t\mathbf{,}x),x\text{ }\mathbf{\in R}^{+},\text{ }t>0, \\ 
u(0,x)=u_{0}(x),x\text{ }\mathbf{\in R}^{+}, \\ 
\mathcal B u\left( t,0\right) =h\left( t\right) ,\text{ }t>0%
\end{array}%
\right.  \label{1.1-1}
\end{equation}%
with inhomogeneous mixed data $h\left( t\right) $, where $L=i\partial
_{t}+\frac{1}{2}\Delta ,$ $\Delta =\partial _{x}^{2},$ $\mathcal{B}=1+\alpha \partial_x$ and $f(t\mathbf{,}x)$ is
the power nonlinearity such that 
\begin{equation*}
f(t\mathbf{,}x)=\lambda \left\vert u\right\vert ^{p-1}u,\lambda \in \mathbf{%
C,}
\end{equation*}
We assume that $u_{0}\left( 0\right) =0$. In this case, compatibility
condition says that $h\left( 0\right) =0.$ We find that solutions of (\ref%
{1.1-1}) are represented as 
\begin{equation*}
u\left( t\right) =w\left( t\right) +z\left( t\right) ,
\end{equation*}%
where $w\left( t\right) $ is the solution of the homogeneous mixed
boundary value problem and $z\left( t\right) $ is the solution of the
inhomogeneous mixed boundary value problem with zero initial data. More
precisely, they are written explicitly through $\left( u_{0}\left( x\right)
,h\left( t\right) \right) $ as 
\begin{equation*}
w\left( t\right) =U\left( t\right) u_{0}-i\int_{0}^{t}U\left( t-\tau
\right) f(\tau )d\tau ,
\end{equation*}%
\begin{equation*}
U\left( t\right) \phi =\frac{1}{\sqrt{2\pi it}}\int_{0}^{\infty }\left(
e^{\frac{i\left\vert x-y\right\vert ^{2}}{2t}}-e^{\frac{i\left\vert
x+y\right\vert ^{2}}{2t}}\right) \phi \left( y\right) dy,
\end{equation*}%
\begin{equation*}
z\left( t,x\right) =\frac{i}{\pi}\mathcal B^{-1}F_s\{e^{\lambda p^2 t}p \hat{h}(\lambda p^2)\}
\end{equation*}%
for $x\mathbf{\in R}^{+},t>0$, (see Section \ref{Section 2}).

This paper is a continuation of the work carried in \cite{EsHaKa}  in which the inhomogeneous Dirichlet- boundary value problem has been considered and sufficient conditions are given to show asymptotic behavior of solutions have been presented.

The Cauchy problem for the cubic nonlinear Schr\"{o}dinger equations was
studied by many authors extensively, see \cite{Cazenave} and references
cited therein. On the other hand, there are some results on the initial
boundary value problem for nonlinear Schr\"{o}dinger equations with
homogeneous boundary conditions (see \cite{BrezisGal1980}, \cite{Ogawa1990}, 
\cite{OgawaOzawa1991}, \cite{Hayashi89-1}, \cite{Hayashi89-2}, \cite%
{Hayashi90}, \cite{Hayashi94}, \cite{Y.Tsutsumi1983}). There are also some
results on the inhomogeneous boundary value problem, see \cite{Bu2000}, \cite%
{CarrollBu1991},  in one dimension and \cite{StraussBu2001}
in existence of weak solutions in general space dimension without uniqueness
of solutions.  However there are few results on the asymptotic behavior of
solutions in the case of inhomogeneous boundary value problem except , \cite{Kaikina} in which the weighted Sobolev spaces are used
to get global results for (\ref{1.1-1}) when $p=3$.

In this paper, we show
local in time existence results of solutions to (\ref{1.1-1}) which is not
stated in , explicitly. It seems that although local
existence of solutions in the usual Sobolev spaces for (\ref{1.1-1}) is
known but in the weighted Sobolev spaces are not well known.
Our another
purpose in this paper is to show the classical energy method \ and
factorization techniques works well for proofs of global existence in time
of solutions to (\ref{1.1-1}). We obtain estimates of Green operator,  which is a modification of the free Schrödinger evolution group, as follows
$$\B^{-1}\F_s e^{ip^2t}\F_s\mathcal B\phi ,$$
where $\F_s$ is a Fourier sine transform and operator $B = 1 + \alpha \partial_x$, morevoer we show that the operator $
J=x+it\partial _{x}$ works well to inhomogeneous cases which is not shown
before and our results recover the previous results obtained in  and the decay conditions on the boundary data are improved due
to more regularity conditions on boundary data. We note that there are also
some results in one dimension using inverse scattering techniques, \cite%
{Focas}, \cite{Fokas 2005}.We note that there are also some results in one
dimension using inverse scattering techniques, \cite{Focas}, \cite{Fokas
2005}. For example, in paper \cite{Fokas2018} local well-posedness of the
initial boundary value problem NLS with data ($u_{0}(x),h(t)$) in $\left(
H_{x}^{s}(R^{+}),H_{t}^{2s+1}(0,T)\right) $ for $s>\frac{1}{2}$ was
established via the unified transform method and a contraction mapping
approach..

We first state a local existence in time of solutions.

\begin{theorem}
\label{Theorem 2}We assume that $p>2.$ Suppose 
\begin{equation*}
u_{0}\in \mathbf{H}^{1,0} \cap \mathbf{H}%
^{0,1} ,h\in \mathbf{C}^{2}\left( \left[ 0,T%
\right] \right) 
\end{equation*}%
and $u_{0}\left( 0\right) =h\left( 0\right) =\partial _{t}h\left( 0\right)
=0.$ Then there exists a time  $T$ $\leq $ $\left( C\rho +1\right) ^{-\frac{4%
}{3}}$ such that (\ref{1.1-1}) has a unique local solution 
\begin{equation*}
u\in \mathbf{C}\left( \left[ 0,T\right] ;\mathbf{H}^{1,0}\left( \mathbf{R}%
^{+}\right) \cap \mathbf{H}^{0,1} \right) ,
\end{equation*}%
where 
\begin{equation*}
\mathbf{H}^{1,0} =\left\{ \phi \in \mathbf{L}%
^{2} ;\left\Vert \phi \right\Vert _{\mathbf{H}%
^{1,0} }=\left\Vert \phi \right\Vert _{\mathbf{L}%
^{2} }+\left\Vert \partial _{x}\phi \right\Vert
_{\mathbf{L}^{2}}<\infty \right\} 
\end{equation*}%
and 
\begin{equation*}
\mathbf{H}^{0,1} =\left\{ \phi \in \mathbf{L}%
^{2} ;\left\Vert \phi \right\Vert _{\mathbf{H}%
1 }=\left\Vert \phi \right\Vert _{\mathbf{L}%
^{2} }+\left\Vert x\phi \right\Vert _{\mathbf{L}%
^{2} }<\infty \right\} .
\end{equation*}
\end{theorem}

Next result says global in time existence and time decay estimates of small
solutions.

\begin{theorem}
\label{Theorem 4}We assume that $\lambda \in \mathbf{R,}$ $p=3$. Suppose 
\begin{equation*}
u_{0}\in \mathbf{H}^{1,0} \cap \mathbf{H}%
^{0,1} ,h\in \mathbf{C}^{2}\left( \left[ 0,T%
\right] \right)
\end{equation*}
and $u_{0}\left( 0\right) =h\left( 0\right) =\partial _{t}h\left( 0\right)
=0.$ Then there exists an $\varepsilon >0$ such that (\ref{1.1-1}) has a
unique global solution 
\begin{equation*}
u\in \mathbf{C}\left( \left[ 0,\infty \right) ;\mathbf{H}^{1,0}\left( 
\mathbf{R}^{+}\right) \cap \mathbf{H}^{0,1}
\right)
\end{equation*}
and time decay estimate 
\begin{equation*}
\left\| u\left( t\right) \right\| _{\mathbf{L}^{\infty }\left( \mathbf{R}%
^{+}\right) \cap \mathbf{C}\left( \left[ 0,\infty \right) \right) }\leq
C\left\langle t\right\rangle ^{-\frac{1}{2}}
\end{equation*}
for any data satisfying 
\begin{equation*}
\left\| u_{0}\right\| _{\mathbf{H}^{1,0} \cap 
\mathbf{H}^{0,1} }\leq \varepsilon
\end{equation*}
and 
\begin{equation*}
\left| h\left( t\right) \right| \leq \varepsilon \left\langle t\right\rangle
^{-\frac{3}{4}-\gamma },\left| \partial _{t}h\left( t\right) \right| \leq
\varepsilon \left\langle t\right\rangle ^{-\frac{7}{4}-\gamma },\left|
\partial _{t}^{2}h\left( t\right) \right| \leq \varepsilon \left\langle
t\right\rangle ^{-1-\gamma },
\end{equation*}
for any positive $\gamma \geq \varepsilon ^{\frac{1}{3}}.$
\end{theorem}

\begin{theorem}
\label{Theorem 6}Let $u\left( t,x\right) $ be the solution constructed in
Theorem \ref{Theorem 4}. Then for any data $\left( u_{0},h\left( t\right)
\right) $, there exists a unique function $\Psi \in \mathbf{L}^{\infty
} \cap \mathbf{C}\left( \left[ 0,\infty \right)
\right) $ such that $\Psi \left( 0\right) =0$%
\begin{equation*}
u\left( t,x\right) =e^{\frac{i\left| x\right| ^{2}}{2t}}\frac{1}{\sqrt{it}}%
\Psi \left( \frac{x}{t}\right) e^{-i\lambda \left| \Psi _{+}\left( \frac{x}{t%
}\right) \right| ^{2}\log t}+O\left( \varepsilon ^{2}t^{-\varepsilon }\left(
1+\log t\right) \right) .
\end{equation*}
\end{theorem}

From the above theorem we find that the asymptotic behavior of solutions is
the same as solutions to the homogeneous boundary value problem. Namely the
boundary data are not effective in the asymptotics of solutions. As we will show
below, our conditions on the boundary data imply z (t) decays in time in the same
order as that of w (t).
In the following theorem, we consider the cases that the order of nonlinearity
$p > 3$ and the weaker time decay condition on the boundary data.

We organize our paper as follows. In Section \ref{Section 1}, we state time
decay estimates through the operator $J=x+it\partial _{x}.$ Section \ref%
{Section 2} is devoted to the proof of local existence Theorem \ref{Theorem
2}. Theorem \ref{Theorem 4} is shown in Section \ref{Section 4} by showing
a-priori estimates of local solutions obtained in Theorem \ref{Theorem 2}.
Finally in Section \ref{Section 5}, we give the proof of Theorem \ref%
{Theorem 6}.

\begin{theorem}
\label{Theorem 7}We assume that $\lambda \in \mathbf{C,}$ $p>3,\partial
_{x}u_{0}\left( 0\right) =h\left( 0\right) =0$ and 
\begin{equation*}
\left| h\left( t\right) \right| \leq \varepsilon \left\langle t\right\rangle
^{-\beta },\left| \partial _{t}h\left( t\right) \right| \leq \varepsilon
\left\langle t\right\rangle ^{-1-\beta },\frac{1}{2}+\frac{1}{p-1}\leq \beta
<1.
\end{equation*}
Then there exists an $\varepsilon >0$ such that (\ref{1.1-1}) has a unique
global solution 
\begin{equation*}
u\in \mathbf{C}\left( \left[ 0,\infty \right) ;\mathbf{H}^{2,0}\left( 
\mathbf{R}^{+}\right) \cap \mathbf{H}^{0,1}
\right)
\end{equation*}
and time decay estimate 
\begin{equation*}
\left\| u\left( t\right) \right\| _{\mathbf{L}^{\infty }\left( \mathbf{R}%
^{+}\right) \cap \mathbf{C}\left( \left[ 0,\infty \right) \right) }\leq
C\left\langle t\right\rangle ^{\frac{1}{2}-\beta }
\end{equation*}
for any data satisfying 
\begin{equation*}
0<\left\| u_{0}\right\| _{\mathbf{H}^{2,0} \cap 
\mathbf{H}^{0,1} }\leq \varepsilon .
\end{equation*}
\end{theorem}

Theorem \ref{Theorem 7} suggests us the time decay conditions on the
boundary data are effective to the time decay of solutions. Indeed we have

\begin{theorem}
\label{Theorem 8} Let $u$ be the time global solutions constructed in
Theorem \ref{Theorem 7}. We also assume that $\frac{1}{2}+\frac{1}{p-1}%
<\beta <1$ and for some constant $A$ 
\begin{equation*}
h\left( t\right) =A\frac{t}{(1+t)^{\beta +1}}+O\left( \frac{t}{\left\langle
t\right\rangle ^{2+\gamma }}\right) ,\gamma >0.
\end{equation*}
Then we have 
\begin{equation*}
u\left( t,x\right) =At^{\frac{1}{2}-\beta }\Lambda (xt^{-\frac{1}{2}%
})+O\left( \varepsilon t^{1-2\beta +\frac{1}{p-1}}\right) ,
\end{equation*}
where $\Lambda (\xi )\in L^{\infty }(R^{+})$ given by formula 
\begin{equation*}
\Lambda (\xi )=\frac{1}{i\sqrt{2i\pi }}\int_{0}^{1}e^{\frac{i\xi ^{2}}{2(1-y)%
}}\frac{1}{y^{\beta }\sqrt{1-y}}dy.
\end{equation*}
\end{theorem}
\section{Preliminary estimates\label{Section 1}}

We begin this section by introducing the notation needed in this work. We use $\|\cdot\|_{\esp^p}$ to denote the $\esp^p(\mathbb R^+)$ norm. If necessary, we use subscript to inform which variable we are
concerned with.  The mixed norm $\esp ^p \mathbf L_T^{\infty}$ of $f(x,t)$ is defined as 
$$\|f\|_{\esp ^p \mathbf L_T^{\infty}}=\sup \limits_{t\in T}\|f(t)\|_{\esp^p}.$$

We need time decay estimate of solutions through the operator 
\begin{equation*}
J=x+it\partial _{x}=ite^{\frac{ix^{2}}{2t}}\partial _{x}e^{-\frac{ix^{2}}{2t}%
}.
\end{equation*}

\begin{lemma}
\label{Lemma 1.1}Let $\phi \in \mathbf{H}^{1,0}
\cap \mathbf{H}^{0,1} ,$ then we have 
\begin{equation*}
\left\| \phi \right\| _{\mathbf{L}^{\infty }}\leq \left\{ 
\begin{array}{c}
Ct^{-\frac{1}{2}}\left\| J\phi \right\| _{\mathbf{L}^{2}}^{\frac{1}{2}}\left\| \phi \right\| _{\mathbf{L}^{2} }^{\frac{1}{2}} \\ 
C\left\| \partial _{x}\phi \right\| _{\mathbf{L}^{2} }^{\frac{1}{2}}\left\| \phi \right\| _{\mathbf{L}^{2} }^{\frac{1}{2}}%
\end{array}
\right. .
\end{equation*}
\end{lemma}

\begin{proof}
By integration by parts, we find that 
\begin{eqnarray*}
\left| \phi \left( x\right) \right| ^{2} &=&-\int_{x}^{\infty }\partial
_{x}\left| \phi \left( x\right) \right| ^{2}dx=-\int_{x}^{\infty }\partial
_{x}\left| e^{-\frac{ix^{2}}{4t}}\phi \left( x\right) \right| ^{2}dx \\
&=&-\int_{x}^{\infty }\partial _{x}e^{-\frac{ix^{2}}{4t}}\phi \left(
x\right) \cdot \overline{e^{-\frac{ix^{2}}{4t}}\phi \left( x\right) }+e^{-%
\frac{ix^{2}}{4t}}\phi \left( x\right) \cdot \overline{\partial _{x}e^{-%
\frac{ix^{2}}{4t}}\phi \left( x\right) }dx \\
&=&\frac{1}{it}\int_{x}^{\infty }J\phi \left( x\right) \cdot \overline{e^{-%
\frac{ix^{2}}{4t}}\phi \left( x\right) }-e^{-\frac{ix^{2}}{4t}}\phi \left(
x\right) \cdot \overline{J\phi \left( x\right) }dx.
\end{eqnarray*}
Hence we have 
\begin{equation*}
\left\| \phi \right\| _{\mathbf{L}^{\infty }}\leq Ct^{-\frac{1}{2}}\left\| J\phi \right\| _{\mathbf{L}^{2} }^{\frac{1}{2}}\left\| \phi \right\| _{\mathbf{L}^{2} }^{\frac{1}{2}}.
\end{equation*}
In the same way we also have the second estimate.
\end{proof}

\begin{lemma}
\label{Lemma 1.2}We let $v\in \mathbf{L}^{2}
\cap \mathbf{H}^{0,1} $ and $f\left( t\right)
=\left| v\right| ^{p-1}v,p\geq 1,$ then 
\begin{equation*}
\left\| f\left( t\right) \right\| _{\mathbf{L}^{2}\left( \mathbf{R}%
^{+}\right) }\leq Ct^{-\frac{p-1}{2}}\left\| Jv\right\| _{\mathbf{L}%
^{2} }^{\frac{p-1}{2}}\left\| v\right\| _{%
\mathbf{L}^{2}}^{\frac{p-1}{2}}\left\|
v\right\| _{\mathbf{L}^{2}},
\end{equation*}
\begin{equation*}
\left\| \partial _{x}f\left( t\right) \right\| _{\mathbf{L}^{2}\left( 
\mathbf{R}^{+}\right) }\leq \left\{ 
\begin{array}{c}
Ct^{-\frac{p-1}{2}}\left\| Jv\right\| _{\mathbf{L}^{2}\left( \mathbf{R}%
^{+}\right) }^{\frac{p-1}{2}}\left\| v\right\| _{\mathbf{L}^{2}\left( 
\mathbf{R}^{+}\right) }^{\frac{p-1}{2}}\left\| \partial _{x}v\right\| _{%
\mathbf{L}^{2}} \\ 
C\left\| \partial _{x}v\right\| _{\mathbf{L}^{2}
}^{\frac{p-1}{2}}\left\| v\right\| _{\mathbf{L}^{2}\left( \mathbf{R}%
^{+}\right) }^{\frac{p-1}{2}}\left\| \partial _{x}v\right\| _{\mathbf{L}%
^{2} }%
\end{array}
\right. ,
\end{equation*}
\begin{equation*}
\left\| \partial _{t}f\left( t\right) \right\| _{\mathbf{L}^{2}\left( 
\mathbf{R}^{+}\right) }\leq \left\{ 
\begin{array}{c}
Ct^{-\frac{p-1}{2}}\left\| Jv\right\| _{\mathbf{L}^{2}\left( \mathbf{R}%
^{+}\right) }^{\frac{p-1}{2}}\left\| v\right\| _{\mathbf{L}^{2}\left( 
\mathbf{R}^{+}\right) }^{\frac{p-1}{2}}\left\| \partial _{t}v\right\| _{%
\mathbf{L}^{2}} \\ 
C\left\| \partial _{x}v\right\| _{\mathbf{L}^{2}
}^{\frac{p-1}{2}}\left\| v\right\| _{\mathbf{L}^{2}\left( \mathbf{R}%
^{+}\right) }^{\frac{p-1}{2}}\left\| \partial _{t}v\right\| _{\mathbf{L}%
^{2} }%
\end{array}
\right.
\end{equation*}
\end{lemma}

and 
\begin{equation*}
\left\| Jf\left( t\right) \right\| _{\mathbf{L}^{2}\left( \mathbf{R}%
^{+}\right) }\leq \left\{ 
\begin{array}{c}
Ct^{-\frac{p-1}{2}}\left\| v\right\| _{\mathbf{L}^{2}\left( \mathbf{R}%
^{+}\right) }^{\frac{p-1}{2}}\left\| Jv\right\| _{\mathbf{L}^{2}\left( 
\mathbf{R}^{+}\right) }^{\frac{p+1}{2}} \\ 
C\left\| v\right\| _{\mathbf{L}^{2}}^{\frac{p-1%
}{2}}\left\| \partial _{x}v\right\| _{\mathbf{L}^{2}\left( \mathbf{R}%
^{+}\right) }^{\frac{p-1}{2}}\left\| Jv\right\| _{\mathbf{L}^{2}\left( 
\mathbf{R}^{+}\right) }%
\end{array}
\right. .
\end{equation*}

\begin{proof}
By a direct calculation 
\begin{equation*}
Jf=J\left| v\right| ^{p-1}v=\frac{p+1}{2}\left| v\right| ^{p-1}Jv-\frac{p-1}{%
2}\left| v\right| ^{p-3}v^{2}\overline{Jv},
\end{equation*}
\begin{equation*}
\partial _{x}f=\left| v\right| ^{p-1}\partial _{x}v+\frac{p-1}{2}\left|
v\right| ^{p-3}\left( \partial _{x}\left| v\right| ^{2}\right) v
\end{equation*}
Hence 
\begin{eqnarray*}
\left\| Jf\right\| _{\mathbf{L}^{2}} &\leq
&C\left\| v\right\| _{\mathbf{L}^{\infty }
}^{p-1}\left\| Jv\right\| _{\mathbf{L}^{2}}, \\
\left\| \partial _{x}f\right\| _{\mathbf{L}^{2}
} &\leq &C\left\| v\right\| _{\mathbf{L}^{\infty }\left( \mathbf{R}%
^{+}\right) }^{p-1}\left\| \partial _{x}v\right\| _{\mathbf{L}^{2}\left( 
\mathbf{R}^{+}\right) }, \\
\left\| \partial _{t}f\right\| _{\mathbf{L}^{2}
} &\leq &C\left\| v\right\| _{\mathbf{L}^{\infty }\left( \mathbf{R}%
^{+}\right) }^{p-1}\left\| \partial _{t}v\right\| _{\mathbf{L}^{2}\left( 
\mathbf{R}^{+}\right) }.
\end{eqnarray*}
We have by Lemma \ref{Lemma 1.1} 
\begin{equation*}
\left\| Jf\right\| _{\mathbf{L}^{2}}\leq
\left\{ 
\begin{array}{c}
Ct^{-\frac{p-1}{2}}\left\| Jv\right\| _{\mathbf{L}^{2}\left( \mathbf{R}%
^{+}\right) }^{\frac{p+1}{2}}\left\| v\right\| _{\mathbf{L}^{2}\left( 
\mathbf{R}^{+}\right) }^{\frac{p-1}{2}} \\ 
C\left\| v\right\| _{\mathbf{L}^{2}}^{\frac{p-1%
}{2}}\left\| \partial _{x}v\right\| _{\mathbf{L}^{2}\left( \mathbf{R}%
^{+}\right) }^{\frac{p-1}{2}}\left\| Jv\right\| _{\mathbf{L}^{2}\left( 
\mathbf{R}^{+}\right) }%
\end{array}
\right. .
\end{equation*}
In the same way as in the proof of the last estimate, we have other
estimates and so we omit their proofs. This completes the proof of the lemma.
\end{proof}

\begin{lemma}
 \label{Lemma 1.3}
 The operator $\mathcal B^{-1}:=(1+\alpha \partial_x)^{-1}$ satisfies
$$\|\mathcal{B}^{-1}\phi\|_{\esp^ \infty}\leq C \|\phi\|_{\esp^ 2}.$$
$$\mathcal{B}^{-1}-1=-\alpha \mathcal{B}^{-1}\partial_x$$
\end{lemma}
\begin{proof}
By virtue of
$$\mathcal{B}^{-1} \mathcal{F}_s =\mathcal{F}^{-1}\frac{1}{1+i\alpha p},$$
we have
$$\mathcal{B}^{-1}\phi=\mathcal{B}^{-1}\mathcal{F}_s\mathcal{F}_s \phi=\mathcal{F}^{-1}\frac{1}{1+i\alpha p}\mathcal{F}_s\phi,$$
Therefore, from the H\"older inequality and the Plancherel theorem, we obtain
\begin{equation}
\|\mathcal{B}^{-1}\phi\|_{\esp^\infty}\leq C \left\|\frac{1}{1+i\alpha p}\right\|_{\esp^2}\|\mathcal{F}_s\phi\|_{\esp^ 2}\leq C\|\phi\|_{\esp ^2}.
\end{equation}
On the other hand, from $\mathcal{B}^{-1}\mathcal{B}=1$ we conclude $\mathcal{B}^{-1}-1=-\alpha \mathcal{B}^{-1}\partial_x$.
\end{proof}
\section{Local existence in time of solutions\label{Section 2}}

We consider the linearized equation of (\ref{1.1-1}) 
\begin{equation}
\left\{ 
\begin{array}{c}
Lu^{\left( n+1\right) }=\lambda \left| u^{\left( n\right) }\right|
^{p-1}u^{\left( n\right) },x\text{ }\mathbf{\in R}^{+},t>0, \\ 
u^{\left( n+1\right) }(0,x)=u_{0}(x),\text{ }x\text{ }\mathbf{\in R}^{+}, \\ 
u^{\left( n+1\right) }\left( t,0\right) +\alpha u^{\left( n+1\right) }_x\left( t,0\right) =h\left( t\right) ,\text{ }%
\end{array}
\right.  \label{2.2}
\end{equation}
for $n\geq 2$ and 
\begin{equation*}
u^{\left( 1\right) }\in \mathbf{C}\left( \left[ 0,\infty \right) ;\mathbf{H}%
^{1,0} \cap \mathbf{H}^{0,1} \right)
\end{equation*}
is the solution of 
\begin{equation*}
\left\{ 
\begin{array}{c}
Lu^{\left( 1\right) }=0,x\mathbf{\in R}^{+},t>0, \\ 
u^{\left( 1\right) }(0,x)=u_{0}(x),\text{ }x\text{ }\mathbf{\in R}^{+}, \\ 
u^{\left( 1\right) }\left( t,0\right) +\alpha u^{\left( 1\right) }_x\left( t,0\right)=h\left( t\right) .%
\end{array}
\right.
\end{equation*}
We assume that the compatibility condition $u^{\left( n+1\right) }\left(
0,0\right) =u_{0}^{\left( n+1\right) }\left( 0\right) =h\left( 0\right) =0$
holds.

For simplicity, we let $u^{\left( n+1\right) }=v,$ $f=\lambda \left|
u^{\left( n\right) }\right| ^{p-1}u^{\left( n\right) }$. We divide $%
u^{\left( n+1\right) }$ into two parts 
\begin{equation*}
u^{\left( n+1\right) }=v=w+z,
\end{equation*}
where $w$ is the solution of homogeneous boundary condition such that 
\begin{equation}
\left\{ 
\begin{array}{c}
Lw=f,\text{ }x\text{\ }\mathbf{\in R}^{+},\text{ }t>0, \\ 
w(0,x)=u_{0}(x),x\mathbf{\in R}^{+}, \\ 
w\left( t,0\right)+\alpha w_x(t,0) =0%
\end{array}
\right.  \label{2.3}
\end{equation}
and $z$ is the solution with zero initial data such that 
\begin{equation}
\left\{ 
\begin{array}{c}
Lz=0,x\mathbf{\in R}^{+},\text{ }t>0, \\ 
z(0,x)=0,\text{ \ }x\text{ \ }\mathbf{\in R}^{+}, \\ 
z\left( t,0\right) +\alpha z_x(t,0)=h\left( t\right) .%
\end{array}
\right.  \label{2.4}
\end{equation}

Multiplying both sides of (\ref{2.3}) and (\ref{2.4}) by $\overline{w}$ and $%
\overline{z},$ respectively, integrating in space and taking the imaginary
part, we obtain 
\begin{eqnarray*}
&&\frac{1}{2}\frac{d}{dt}\left\| w\left( t\right) \right\| _{\mathbf{L}%
^{2} }^{2} \\
&=&-\frac{1}{2}\text{Im}\int_{0}^{\infty }\partial _{x}\left( \partial
_{x}w(t,x)\cdot \overline{w}(t\mathbf{,}x)\right) dx+\text{Im}%
\int_{0}^{\infty }f(t\mathbf{,}x)\overline{w}(t\mathbf{,}x)dx
\end{eqnarray*}
and 
\begin{equation*}
\frac{1}{2}\frac{d}{dt}\left\| z\left( t\right) \right\| _{\mathbf{L}%
^{2} }^{2}=-\frac{1}{2}\text{Im}\int_{0}^{\infty
}\partial _{x}\left( \partial _{x}z(t,x)\cdot \overline{z}(t\mathbf{,}%
x)\right) dx.
\end{equation*}
By the boundary conditions, since $\alpha \in \mathbb R$, we find that 
\begin{equation*}
\text{Im } \int_{0}^{\infty }\partial _{x}\left( \partial _{x}w(t,x)\cdot \overline{w}(t%
\mathbf{,}x)\right) dx=-\alpha \text{Im } \left( \partial_x w(t,0)\overline{\partial_x w(t,0)}\right)=0,
\end{equation*}
\begin{equation*}
\text{Im } \int_{0}^{\infty }\partial _{x}\left( \partial _{x}z(t,x)\cdot \overline{z}(t%
\mathbf{,}x)\right) dx=-\text{Im }\left( {\partial_x z}(t,0)\cdot\overline{h(t)}\right)
\end{equation*}
Hence we obtain 
\begin{equation}
\left\| w\left( t\right) \right\| _{\mathbf{L}^{2} }^2\leq \left\| u_{0}\right\| _{\mathbf{L}^{2} }^2+\int_{0}^{t}\left\| f\left( \tau \right) \right\| _{\mathbf{L}^{2}}\left\| w\left( \tau \right) \right\| _{\mathbf{L}^{2}}d\tau  \label{2.5}
\end{equation}
and 
\begin{equation}
\left\| z\left( t\right) \right\| _{\mathbf{L}^{2}\left( \mathbf{R}%
^{+}\right) }^{2}\leq \int_{0}^{t}\left|\partial_x z\left( \tau ,0\right) \right|
\left| h\left( \tau \right) \right| d\tau .  \label{2.6}
\end{equation}

We differentiate (\ref{2.3}) and (\ref{2.4}) with respect to time $t$ and in
the same way as in the proofs of (\ref{2.5}) and (\ref{2.6}), we get with
the identity $i\partial _{t}w\left( 0\right) =-\frac{1}{2}\partial
_{x}^{2}u_{0}+\lambda f\left( 0,x\right) $
\begin{eqnarray}
&&\left\| \partial _{t}w\left( t\right) \right\| _{\mathbf{L}^{2}\left( 
\mathbf{R}^{+}\right) }  \notag \\
&\leq &\left\| \partial _{t}w\left( 0\right) \right\| _{\mathbf{L}^{2}\left( 
\mathbf{R}^{+}\right) }+\int_{0}^{t}\left\| \partial _{t}f\left( \tau
\right) \right\| _{\mathbf{L}^{2} }\|\partial_\tau w\|_{\esp^2}d\tau  \notag
\\
&\leq &C\left( \left\| \partial _{x}^{2}u_{0}\right\| _{\mathbf{L}^{2}\left( 
\mathbf{R}^{+}\right) }+\left\| f\left( 0\right) \right\| _{\mathbf{L}%
^{2} }+\int_{0}^{t}\left\| \partial _{\tau
}f\left( \tau \right) \right\| _{\mathbf{L}^{2}
}\|\partial_\tau w\|_{\esp^2}d\tau \right)  \label{2.91}
\end{eqnarray}
and 
\begin{equation}
\left\| \partial _{t}z\left( t\right) \right\| _{\mathbf{L}^{2}\left( 
\mathbf{R}^{+}\right) }^{2}\leq \frac{1}{|\alpha| }\int_{0}^{t}\left| \partial _{\tau }z\left(
\tau ,0\right) \right| \left| \partial _{\tau }h\left( \tau \right) \right|
d\tau .  \label{2.10bb}
\end{equation}

In the same way, we have 
\begin{eqnarray*}
&&\frac{1}{2}\frac{d}{dt}\left\| \partial _{x}w\left( t\right) \right\| _{%
\mathbf{L}^{2}\left( \mathbf{R}^{+}\right) }^{2} \\
&=&-\frac{1}{2}\text{Im}\int_{0}^{\infty }\partial _{x}\left( \partial
_{x}^{2}w(t\mathbf{,}x)\cdot \partial _{x}\overline{w}(t,x)\right) dx+\text{%
Im}\int_{0}^{\infty }\partial _{x}f(t,x)\cdot \partial _{x}\overline{w}%
(t,x)dx
\end{eqnarray*}

\begin{eqnarray*}
\frac{1}{2}\frac{d}{dt}\left\| \partial _{x}z\left( t\right) \right\| _{%
\mathbf{L}^{2} }^{2} 
&=&-\frac{1}{2}\text{Im}\int_{0}^{\infty }\partial _{x}\left( \partial
_{x}^{2}z(t\mathbf{,}x)\cdot \partial _{x}\overline{z}(t,x)\right) dx \\
&=&-\frac{1}{2}\text{Im}\int_{0}^{\infty }\partial _{x}\left( -2i\partial
_{t}z(t\mathbf{,}x)\cdot \partial _{x}\overline{z}(t,x)\right) dx \\
&=&-\text{Re }\left( \partial _{t}z(t\mathbf{,}0)\cdot \overline{\partial_x z(t,0)}\right) ,
\end{eqnarray*}
therefore
\begin{equation}\label{2.7}
\left\| \partial _{x}w\left( t\right) \right\| _{%
\mathbf{L}^{2}\left( \mathbf{R}^{+}\right) }^{2}\leq C\left|  \int \limits_0^t \text{Im }(\partial
_{x}^{2}w(\tau\mathbf{,}0)\cdot \partial _{x}\overline{w}(\tau,0))d \tau\right|  +C \int_{0}^{t}\|\partial_x f(\tau)\|_{\esp^ 2}\|\partial _x w\|_{\esp^ 2}d\tau
\end{equation}
\begin{equation}
\left\| \partial _{x}z\left( t\right) \right\| _{\mathbf{L}^{2}\left( 
\mathbf{R}^{+}\right) }^{2}\leq 2\int_{0}^{t}\left| \partial _{\tau }z\left(
\tau ,0\right) \right| \left| \partial_x z(\tau,0)\right| d\tau .
\label{2.8}
\end{equation}

We multiply both sides of (\ref{2.4}) by $J$ and use the
energy method to obtain 
\begin{eqnarray}
&&\frac{1}{2}\frac{d}{dt}\left\| Jw\left( t\right) \right\| _{\mathbf{L}%
^{2}\left( \mathbf{R}^{+}\right) }^{2}  \notag \\
&=&-\frac{1}{2}\text{Im}\int_{0}^{\infty }\partial _{x}\left( \partial
_{x}Jw(t\mathbf{,}x)\cdot \overline{Jw}(t,x)\right) dx+\text{Im}%
\int_{0}^{\infty }Jf(t,x)\cdot \overline{Jw}(t,x)dx  \label{2.11}
\end{eqnarray}
and
\begin{equation}
\frac{1}{2}\frac{d}{dt}\left\| Jz\left( t\right) \right\| _{\mathbf{L}%
^{2} }^{2}=-\frac{1}{2}\text{Im}\int_{0}^{\infty
}\partial _{x}\left( \partial _{x}Jz(t\mathbf{,}x)\cdot \overline{Jz}%
(t,x)\right) dx.  \label{2.12}
\end{equation}
We observe
We consider the first terms of the right hand side of (\ref{2.11}) and (\ref%
{2.12}) to have 
\begin{eqnarray}
&&\int_{0}^{\infty }\partial _{x}\left( \partial _{x}Jw\cdot \overline{Jw}%
\right) dx  \notag \\
&=&\int_{0}^{\infty }\partial _{x}\left( \left( 1+x\partial _{x}\right)
w\cdot \overline{\left( x+it\partial _{x}\right) w}\right) dx  \notag \\
&&+it\int_{0}^{\infty }\partial _{x}\left( \partial _{x}^{2}w\cdot \overline{%
\left( x+it\partial _{x}\right) w}\right) dx\\
&=&-it w(t,0)\overline{\partial_x w(t,0)}+t^2\partial_x^2 w(t,0)\overline{\partial_x w(t,0)} \label{2.13}
\end{eqnarray}
 and similarly 
\begin{eqnarray*}
&&\int_{0}^{\infty }\partial _{x}\left( \partial _{x}Jz\cdot \overline{Jz}%
\right) dx \\
&=&-it\int_{0}^{\infty }\partial _{x}\left( z\cdot \overline{\partial _{x}z}%
\right) dx+t^{2}\int_{0}^{\infty }\partial _{x}\left( \partial
_{x}^{2}z\cdot \overline{\partial _{x}z}\right) dx.
\end{eqnarray*}
We use the identity $\partial _{x}^{2}z=-2i\partial _{t}z$ to find and 
\begin{equation}
\int_{0}^{\infty }\partial _{x}\left( \partial _{x}Jz\cdot \overline{Jz}%
\right) dx=\left( -itz(t,0)-2it^{2}\partial _{t}z(t,0)
\right) \cdot \overline{\partial _{x}z\left( t,0\right) }.  \label{2.14}
\end{equation}
We apply (\ref{2.13}) and (\ref{2.14}) to (\ref{2.11}) and (\ref{2.12}),
respectively to get 
\begin{equation}
\begin{array}{l}
\left\| Jw\left( t\right) \right\| _{\mathbf{L}%
^{2}\left( \mathbf{R}^{+}\right) }^{2}\leq\displaystyle  C\int \limits_0^t |i\tau w(\tau,0)+\tau^2\partial_x^2 w(\tau,0)||\overline{\partial_x w(t,0)} |d\tau\\
\hspace{3cm}\displaystyle +C \int_{0}^{t
}\|Jf(\tau)\|_{\esp^2} \|Jw(\tau)\|_{\esp^ 2}d\tau  \label{2.15}
\end{array}
\end{equation}
and 
\begin{equation}\label{2.140}
\left\Vert Jz\left( t\right) \right\Vert _{\mathbf{L}^{2}\left( \mathbf{R}%
^{+}\right) }^{2}\leq \int_{0}^{t}|\tau z(\tau,0)+2i\tau^{2}\partial _{\tau}z(\tau,0)||\partial _{x}z\left( t,0\right)|
\end{equation}

We divide the proof of Theorem \ref{Theorem 2} into several parts. We
introduce the function space 
\begin{equation*}
\mathbf{X}_{T}=\left\{ \phi \left( t\right) \in \mathbf{C}\left( \left[ 0,T%
\right] ;\mathbf{X}\right) ;\left\| \phi \right\| _{\mathbf{X}%
_{T}}=\sup_{t\in \left[ 0,T\right] }\left\| \phi \left( t\right) \right\| _{%
\mathbf{X}}<\infty \right\} ,
\end{equation*}
where 
\begin{eqnarray*}\left\| \phi \left( t\right) \right\| _{\mathbf{X}}^{2} &=&\left\| \phi
\left( t\right) \right\| _{\mathbf{L}^{2}
}^{2}+\left\| \partial _{x}\phi \left( t\right) \right\| _{\mathbf{L}%
^{2} }^{2}+\left\| J\phi \left( t\right)
\right\| _{\mathbf{L}^{2} }^{2} \\
&&+\left\| \partial _{x}^{2}\phi \left( t\right) \right\| _{\mathbf{L}%
^{2} }^{2}+\left\| \partial _{t}\phi \left(
t\right) \right\| _{\mathbf{L}^{2} }^{2}.
\end{eqnarray*}
We first prove
\begin{lemma}\label{Lemma nonhom}
 We assume $h\in \mathbf C([0,T])$, with $h(0)=0$  and 
\begin{equation*}
\sup_{t\in \left[ 0,T\right] }\left| h\left( t\right) \right| +\sup_{t\in \left[ 0,T\right] }\left| \partial _{t}h\left( t\right) \right| \leq \rho^\rho .
\end{equation*}
 Then the solution of \eqref{2.4} is given by
 \begin{equation}
  z(t,x)=\mathcal{B}^{-1}\left\{\frac{1}{\sqrt{2i\pi }}\int_{0}^{t}e^{\frac{ix^{2}}{2\tau }}\frac{x}{\tau\sqrt{\tau }}h\left( t-\tau \right) d\tau .\right\} 
\end{equation}  
 Furthermore we have the estimate such that 
\begin{equation*}
\left\| z\right\| _{\mathbf{X}_{T}}\leq C\rho ^{p}T^{\frac{3}{2}}\left( 1+T^{\frac{3}{2}}+T^{\frac{5}{2}}\right) .
\end{equation*}
\end{lemma}
\begin{proof}
By changing of variable$\frac{x}{\sqrt{\tau }}=y$%
\begin{equation*}
z\left( t,x\right) =\mathcal{B}^{-1}\left\{\frac{2}{\sqrt{2i\pi }}\int_{\frac{x}{\sqrt{t}}}^{\infty
}e^{\frac{iy^{2}}{2}}h\left( t-\frac{x^{2}}{y^{2}}\right) dy\right\}
\end{equation*}
which implies 
\begin{equation*}
\mathcal B z\left( t,0\right) =\frac{2}{\sqrt{2i\pi }}h\left( t\right) \int_{0}^{\infty
}e^{\frac{iy^{2}}{2}}dy=h\left( t\right) .
\end{equation*}
Using the commutator relations $[\partial_x, \mathcal{B}^{-1}]=0$ and by a direct calculation with $h\left( 0\right) =0$ we have
\begin{equation*}
 \begin{array}{rcl}
   \partial_x z(t,x)\hspace{-0.3cm}&=&\displaystyle\mathcal{B}^{-1}\partial_x \left\{\frac{2}{\sqrt{2i\pi }}\int_{\frac{x}{\sqrt{t}}}^{\infty
}e^{\frac{iy^{2}}{2}}h\left( t-\frac{x^{2}}{y^{2}}\right) dy\right\}\\
\hspace{-0.3cm}&=&\mathcal{B}^{-1}\left\{\displaystyle\frac{2}{\sqrt{2i\pi }}\int_{0}^{t}e^{\frac{ix^{2}}{2\tau }}\frac{1}{\sqrt{\tau }}\partial _{t}h\left( t-\tau \right) d\tau \right\},  
 \end{array}
\end{equation*}
\begin{equation*}
\partial _{x}^{2}z\left( t,x\right) =\mathcal{B}^{-1}\left\{\frac{2i}{\sqrt{2i\pi }}\int_{0}^{t}e^{%
\frac{ix^{2}}{2\tau }}\frac{x}{\tau \sqrt{\tau }}\partial _{t}h\left( t-\tau
\right) d\tau\right\}
\end{equation*}
and 
\begin{equation*}
\partial _{t}z\left( t,x\right) =-\mathcal{B}^{-1}\left\{\frac{1}{\sqrt{2i\pi }}\int_{0}^{t}e^{\frac{%
ix^{2}}{2\tau }}\frac{x}{\tau \sqrt{\tau }}\partial _{t}h\left( t-\tau
\right) d\tau \right\}.
\end{equation*}
Therefore we find that $z$ is the solution of 
\begin{equation*}
i\partial _{t}z+\frac{1}{2}\partial _{x}^{2}z=0
\end{equation*}
with $\mathcal B z\left( t,0\right) =h\left( t\right) .$  

To estimate $z(t,0), \ \partial_x(t,0)$ and $\partial_tz(t,0)$ we use
$$e^{\frac{ix^{2}}{2\tau }}\frac{x}{\tau\sqrt{\tau }}=\mathcal{F}_s\{e^{i\frac{p^2}{2}\tau}p\},$$
to rewrite $z$ as
\begin{equation}\label{bfz}
z(t,x)=\mathcal B^{-1}\mathcal F_s \left\{p\int\limits_0^te^{i\frac{p^2}{2}\tau} h(t-\tau)d\tau\right\}.
\end{equation}
By virtue of 
$$\mathcal B^{-1}\mathcal{F}_s=\mathcal{F}^{-1}\frac{1}{1+i\alpha p}$$
we have
\begin{equation}
 \begin{array}{rcl}
   z(t,0)\hspace{-0.3cm}&=&\hspace{-0.3cm}\displaystyle \int \limits_0^t h(t-\tau) \int \limits_{-\infty}^\infty  e^{i\frac{p^2}{2}\tau}\frac{p}{1+i\alpha p }dp\ d\tau
 \end{array}
\end{equation}
Using $\mathcal{F}^{-1}\{e^{i\frac{p^2}{2}\tau}\}=\frac{\sqrt{\pi}}{\sqrt{i\tau}}, \ \ \ \mathcal F_c\{e^{-y}\} =\frac{1}{1+\alpha^2 p^2},$
we get
$$\int \limits_{-\infty}^\infty  e^{i\frac{p^2}{2}\tau}\frac{1}{1+i\alpha p }dp=\frac{C}{\sqrt{\tau}}\int\limits_{0}^{\infty}e^{-y}e^{\alpha^2 \frac{iy^2}{2\tau}}dy$$
as consequence
$$\int \limits_{-\infty}^\infty  e^{i\frac{p^2}{2}\tau}\frac{p}{1+i\alpha p }dp=\frac{\sqrt{\pi}}{\alpha\sqrt{i\tau}}-\frac{\sqrt{2}}{\sqrt{\tau}}\int\limits_{0}^{\infty}e^{-y}e^{\alpha^2 \frac{iy^2}{2\tau}}dy=\frac{C}{\sqrt{\tau}}$$ thus
\begin{equation}\label{3.13}
 \begin{array}{rcl}
   z(t,0)\hspace{-0.3cm}&=&\hspace{-0.3cm}\displaystyle  C \int \limits_0^t \frac{h(t-\tau)}{\sqrt{\tau}}\ d\tau
 \end{array}
 \end{equation}
 In the same way we have
\begin{equation}\label{3.14}
 \begin{array}{rcl}
   \partial_x z(t,0)\hspace{-0.3cm}&=&\hspace{-0.3cm}\displaystyle C\int \limits_0^t \frac{1}{\sqrt{\tau}}\partial_{\tau}h(t-\tau) \ d\tau
 \end{array}
\end{equation}
\begin{equation}\label{3.15}
 \begin{array}{rcl}
   \partial_tz(t,0)\hspace{-0.3cm}&=&\hspace{-0.3cm}\displaystyle C \int \limits_0^t \frac{1}{\sqrt{\tau}}\partial_{\tau}h(t-\tau) d\tau
 \end{array}
\end{equation}
by the conditions on $h$ we have the estimates 
\begin{equation}
|z(t,0)|+|\partial_tz(t,0)|+\left| \partial _{x}z\left( t,0\right) \right| \leq C\rho^\rho \sqrt{t}.
\label{2.19}
\end{equation}
We apply (\ref{2.19}) to (\ref{2.6}), (\ref{2.8}) and (\ref{2.12}) to get 
\begin{equation}
\left\Vert z\left( t\right) \right\Vert _{\mathbf{L}^{2}\left( \mathbf{R}%
^{+}\right) }^{2}\leq \rho ^{p}\int_{0}^{t}\sqrt{\tau }\left\vert h\left(
\tau \right) \right\vert d\tau \leq \frac{2}{3}C\rho ^{2p}t^{\frac{3}{2}},  \label{l}
\end{equation}
\begin{eqnarray}
\left\Vert \partial _{x}z\left( t\right) \right\Vert _{\mathbf{L}^{2}\left( 
\mathbf{R}^{+}\right) }^{2} &\leq &2 C \rho ^{p}\int_{0}^{t}\tau
d\tau  \leq \rho ^{2p}t^2
\label{l0}
\end{eqnarray}
\begin{eqnarray}
\left\Vert \partial _{t}z\left( t\right) \right\Vert _{\mathbf{L}^{2}\left( 
\mathbf{R}^{+}\right) }^{2} &\leq &2\rho ^{p}\int_{0}^{t}\sqrt{\tau}\left\vert \partial _{\tau }h\left( \tau
\right) \right\vert d\tau  \leq  C \rho ^{2p}t^{\frac{3}{2}},
\label{l1}
\end{eqnarray}
From the equation for $z$ we also have 
\begin{equation*}
\left\Vert \partial _{x}^{2}z\left( t\right) \right\Vert _{\mathbf{L}%
^{2} }\leq 2\left\Vert \partial _{t}z\left(
t\right) \right\Vert _{\mathbf{L}^{2} }.
\end{equation*}
Therefore we get 
\begin{eqnarray}
&&\left\Vert z\left( t\right) \right\Vert _{\mathbf{L}^{2}\left( \mathbf{R}%
^{+}\right) }+\left\Vert \partial _{x}z\left( t\right) \right\Vert _{\mathbf{%
L}^{2} }  \notag 
+\left\Vert \partial _{t}z\left( t\right) \right\Vert _{\mathbf{L}%
^{2} }+\left\Vert \partial _{x}^{2}z\left(
t\right) \right\Vert _{\mathbf{L}^{2} }\leq
3\rho ^{p}t^{\frac{3}{2}}\left( 1+t^{\frac{1}{2}}\right) .  \label{2.18}
\end{eqnarray}
By (\ref{2.14}) and (\ref{2.19}) 
\begin{eqnarray*}
\left\Vert Jz\left( t\right) \right\Vert _{\mathbf{L}^{2}\left( \mathbf{R}%
^{+}\right) }^{2} &\leq &C \int_{0}^{t}\left\vert \tau z\left( \tau ,0\right)
+2\tau ^{2}\partial _{\tau }z\left( \tau ,0\right) \right\vert \left\vert
z_x(\tau,0) \right\vert d\tau \\
&\leq &C \rho ^{2p}\left( \frac{t^3}{3}+2\frac{t^4}{2}\right) .
\end{eqnarray*}

Hence, we have the lemma.
\end{proof}

Let us consider the estimate of solutions to the homogeneous problem \eqref{2.3}. We have

\begin{lemma}
 We assume that\label{Lemma hom.}
 \begin{equation}
  \left\Vert u_{0}\right\Vert _{\mathbf{L}^{2}}^{2}+\left\Vert \partial _{x}u_{0}\right\Vert _{\mathbf{L}^{2} }^{2}+\left\Vert xu_{0}\right\Vert _{\mathbf{L}^{2} }^{2}+\|u_0\|_{\esp^\infty}^2\leq \varepsilon ^{2} 
 \end{equation}
and 
\begin{equation*}
\left\Vert f\left( t\right) \right\Vert _{%
\mathbf{L}^{2}\mathbf{L}^{\infty}_T }+
\left\Vert \partial _{x}f\left( t\right) \right\Vert _{\mathbf{L}^{2}\mathbf{L}^{\infty}_T }+
\left\Vert xf\left(t\right) \right\Vert _{\mathbf{L}^{2}\mathbf{L}^\infty_T}\leq
\delta .
\end{equation*}
Then the solution of \eqref{2.3} is given by 
\begin{equation}\label{t}
 w(t)=U(t)u_0+\int \limits_{0}^{t}U(t-\tau)f(\tau) d\tau,
\end{equation}
where
\begin{equation}\label{UM}
 U(t)\phi=\mathcal B^{-1}\mathcal{F}_se^{\frac{i}{2} p^2t }\mathcal{F}_s\mathcal B \phi, \ \ \mathcal{B}=(1+\alpha \partial_x).  
\end{equation}
 Moreover the solution $w_M$ of \eqref{2.3} satisfies 
 \begin{equation}\label{des}
  \|w\|_{\mathbf X_T}^2\leq 3C\varepsilon^2+C\delta \epsilon T +C \delta T^2.
 \end{equation}

\end{lemma}
\begin{proof}
By a similar method that was applied in \cite{Kaikina} we obtain  that the solution is 
given by \eqref{t}.
We have by (\ref{2.5}), \eqref{2.91} and Lemma \ref{Lemma 1.1}
\begin{equation} \label{L2w}
 \begin{array}{l}
  \left\| w\left( t\right) \right\| _{\mathbf{L}^{2} }^2+\|\partial_t w\|_{\esp^2}^2
\leq \left\| u_{0}\right\| _{\mathbf{L}^{2}}^2+\|\partial^2_x u_0\|_{\esp^2}\\
\hspace{3.5cm}\displaystyle +C\int \limits_{0}^t\left(\|f(\tau)\|_{\esp^{2}}\|w(\tau)\|_{\esp^2}+\|\partial_\tau f(\tau)\|_{\esp^{2}}\|\partial_\tau w(\tau)\|_{\esp^2}\right)d\tau\\
\hspace{3cm}\leq \varepsilon^2 +C\delta\|w\|_{\mathbf X_T}T
 \end{array}
\end{equation}
We also have by the identity $i\partial _{t}w+\frac{1}{2}\partial
_{x}^{2}w=f $ and (\ref{2.91}) 
\begin{equation*}\label{Lwx}
\left\| \partial _{x}^{2}w\left( t\right) \right\| _{\mathbf{L}^{2}}\leq 2\left\| \partial _{t}w\left( t\right) \right\|
_{\mathbf{L}^{2} }+2\delta T\leq \left\|
\partial _{x}^{2}u_{0}\right\| _{\mathbf{L}^{2}
}+3\delta T.
\end{equation*}
To estimate $\|\partial_xw\|_{\esp^2}$ and $\|Jw\|_{\esp^2}$,
we need a estimation for  $w(t,0)$. 
SInce  $|U(t)\phi\mid_{x=0} |\leq\|\phi\|_{\mathbf H^{1}} $ via \eqref{t} we conclude
\begin{equation}
|w(t,0)|\leq \epsilon +\delta t.
\end{equation}
From boundary condition we have 
\begin{equation*}
\text{Im }\left( \partial_x^2 w(\tau,0)\overline{\partial_xw(\tau,0)}\right)=\frac{2}{\alpha}\text{Im} \left(f(\tau,0)\overline{w(\tau,0)} \right)+\frac{1}{\alpha}\frac{d}{dt}|w(\tau,0)|^2,
\end{equation*}
therefore
\begin{equation}
\int\limits_0^t \text{Im }\left( \partial_x^2 w(\tau,0)\overline{\partial_xw(\tau,0)}\right)=\frac{1}{\alpha}|w(\tau,0)|^2+\frac{2}{\alpha}\int\limits_0^t \text{Im } \left(f(\tau,0)\overline{w(\tau,0)} \right)d\tau,
\end{equation}
thus \eqref{2.7} can be rewrite as
\begin{equation}
\begin{array}{l}
\left\| \partial _{x}w\left( t\right) \right\| _{%
\mathbf{L}^{2} }^{2}\\
\displaystyle \leq \frac{C}{|\alpha|}|w(\tau,0)|^2+\frac{2}{\alpha}\int\limits_0^t |f(\tau,0){w(\tau,0)} |d\tau+C \int_{0}^{t}\|\partial_x f(\tau)\|_{\esp^ 2}\|\partial _x w\|_{\esp^ 2}d\tau
\end{array}
\end{equation}
therefore
\begin{equation}
\left\| \partial _{x}w\left( t\right) \right\| _{%
\mathbf{L}^{2} }^{2}\leq C\epsilon^2+C\delta T^2+C\delta T\|w\|_{\mathbf X_T}.
\end{equation}
In the same way we can prove
\begin{equation}
\left\| Jw\left( t\right) \right\| _{\mathbf{L}%
^{2}}^{2}\leq \left\Vert xu_{0}\right\Vert _{\mathbf{L}^{2}\left( 
\mathbf{R}^{+}\right) }^{2}+C\delta\epsilon (T+T^2).
\end{equation}
Therefore we have the lemma.
\end{proof}

We consider the estimate of solutions to \eqref{2.2}.

\begin{lemma}
\label{Lemma 2.5} We assume that
 \begin{equation}
  \left\Vert u_{0}\right\Vert _{\mathbf{L}^{2}}^{2}+\left\Vert \partial _{x}u_{0}\right\Vert _{\mathbf{L}^{2} }^{2}+\left\Vert xu_{0}\right\Vert _{\mathbf{L}^{2} }^{2}+\|u_0\|_{\esp^\infty}^2\leq \varepsilon ^{2} ,
 \end{equation}
 \begin{equation*}
\sup_{t\in \left[ 0,T\right] }\left| h\left( t\right) \right| +\sup_{t\in \left[ 0,T\right] }\left| \partial _{t}h\left( t\right) \right| \leq \rho^p .
\end{equation*}
and
\begin{equation*}
\left\Vert u^{\left( n\right) }\right\Vert _{\mathbf{X}_{T}}\leq
3\varepsilon,
\end{equation*}
with
$$T=O\left(min \left\{\frac{1}{4C|\lambda|p2^{p-1}3^p\varepsilon^{p-1}}, \frac{1}{4C\varepsilon^{\frac{p-1}{2}}},\frac{1}{8C\varepsilon^{2(p-1)}}\right\} \right)$$
Then we have for the solution of \eqref{2.2}
\begin{equation*}
\left\Vert u^{\left( n+1\right) }\right\Vert _{\mathbf{X}_{T}}\leq
3\varepsilon .
\end{equation*}
\end{lemma}

\begin{proof}
 In \cite{EsHaKa} was proved that $f(u)=\lambda|u^{(n)}|^pu^{(n)}$, by Lemma \eqref{Lemma 1.2}  we have 
 $$\|f\|_{\esp^2}+\|\partial_x f\|_{\esp^ 2}+\|\mathcal{J}f\|_{\esp^2}\leq C \|u^{(n)}\|_{\mathbf{X}_T}\leq C|\lambda|p2^{p-1}3^p\varepsilon^p.$$ 
 As our solution of the IBVP can be factored as $u^{(n)}=w^{(n)}+z$, from Lemmas \eqref{Lemma hom.} and \eqref{Lemma nonhom} with $\rho^p=|\lambda|p2^{p-1}3^p\varepsilon^p$ we obtain
 \begin{equation}
 \begin{array}{rcl}
  \|u^{(n+1)}\|_{ \mathbf{X}_{T}}&\leq &\|w^{(n+1)}\|_{\mathbf{X}_T}+\|z\|_{\mathbf{X}_T}\\
  &\leq&2C\varepsilon+C|\lambda|p2^{p-1}3^p\varepsilon^pT+4C\varepsilon^p`T^{\frac{3}{2}}(1+T^{\frac{3}{2}}+T^{\frac{5}{2}})\\
  &<&3\varepsilon
 \end{array}
 \end{equation}
here, we used the assumptions on $T$.
\end{proof}

\begin{lemma}\label{Lemma 2.6}
 Let $\{u^{n+1}\}$ be the solution of  \eqref{2.2} satisfying $\|u^{k}\|_{\mathbf{X}^T}\leq 3\varepsilon$ for any $k$.
 Then there exists a time $T>0$ such that the difference of solutions $ X^{(n+1)}:=u^{(n+1)}-u^{(n)}=w^{(n+1)}-w^{(n)}$ satisfies 
 \begin{equation}
  \|X^{(n+1)}\|_{\mathbf X_T}\leq \frac{1}{2}\|X^{(n)}\|_{\mathbf X_{T}},  \ \ \text{for }n\geq 3.
 \end{equation}
\end{lemma}
The proof of previous Lemma is similar to that found in \cite{EsHaKa}. From Lemmas \eqref{Lemma 2.5},
\eqref{Lemma 2.6}  and the contraction mapping principle, we have Theorem \eqref{Theorem 2}.

\section{Global existence in time of solutions \label{Section 4}}
We divide the solutions into two parts such that $u=w+z$ as in
Section \ref{Section 2}, where 
\begin{equation}
\left\{ 
\begin{array}{c}
Lw=\lambda \left\vert u\right\vert ^{2}u,\text{ }x\text{ }\mathbf{\in R}%
^{+},t>0, \\ 
w(0,x)=u_{0}(x),x\text{ }\mathbf{\in R}^{+}, \\ 
w\left( t,0\right) =0,\text{ }%
\end{array}%
\right.   \label{4.3}
\end{equation}%
\begin{equation}
\left\{ 
\begin{array}{c}
Lz=0,x\mathbf{\in R}^{+},t>0, \\ 
z(0,x)=0,x\text{ }\mathbf{\in R}^{+}, \\ 
z\left( t,0\right) =h\left( t\right) .\text{ }%
\end{array}%
\right.   \label{4.4}
\end{equation}
We introduce the functional space
$$\mathbf X_{T}=\left\{u(t)\in \mathbb{C}\left([0,T);\mathbf{H}^{1,0}  \cap \mathbf{H}^{0,1}(\mathbf{R}^+  )\right) ;\|u\|_{\mathbf{X}_T}<\infty\right\}$$
\begin{eqnarray*}
\left\Vert u\right\Vert _{\mathbf{X}_{T}}^{2} &=&\sup_{t\in \left[
0,T\right) }\left\langle t\right\rangle ^{-2\gamma }\left\Vert u\left(
t\right) \right\Vert _{\mathbf{X}}^{2} \\
&&+\sup_{t\in \left[ 0,T\right) }\left\Vert Ju\left( t\right) \right\Vert _{%
\mathbf{L}^{2} }^{2}\left\langle t\right\rangle
^{-\frac{1}{2}+2\gamma }+\sup_{t\in \left[ 0,T\right) }\left\Vert
u\left( t\right) \right\Vert _{\mathbf{L}^{\infty }\left( \mathbf{R}%
^{+}\right) }^{2}\left\langle t\right\rangle 
\end{eqnarray*}%
and 

\begin{eqnarray*}
\left\| u\left( t\right) \right\| _{\mathbf{X}}^{2} &=&\left\| \phi
\left( t\right) \right\| _{\mathbf{L}^{2}
}^{2}+\left\| \partial _{x}u\left( t\right) \right\| _{\mathbf{L}%
^{2} }^{2}+\left\| \partial _{x}^{2}u\left( t\right) \right\| _{\mathbf{L}%
^{2} }^{2}+\left\| \partial _{t}u\left(
t\right) \right\| _{\mathbf{L}^{2} }^{2}.
\end{eqnarray*}
 By Theorem \ref{Theorem 2}, we note that 
\begin{equation*}
\left\| u\right\| _{\mathbf{X}_{1}}\leq 3\varepsilon
\end{equation*}
if we take $\varepsilon $ small enough. We now prove that for any time $%
\widetilde{T},$ the estimate 
\begin{equation}
\left\| u\right\| _{\mathbf{X}_{\widetilde{T}}}^{2}<\varepsilon ^{\frac{4}{3}%
}  \label{1}
\end{equation}
holds. If the above estimate does not hold, then we can find a finite time $%
T $ such that 
\begin{equation}
\left\| u\right\| _{\mathbf{X}_{T}}^{2}=\varepsilon ^{\frac{4}{3}}.
\label{2}
\end{equation}
However Lemma \ref{Lemma 4.1} and Lemma \ref{Lemma 4.2} below show that $T$
satisfying (\ref{2}) does not exist. This is the desired contradiction.
Namely for any time $\widetilde{T},$ we have (\ref{1}).

\begin{lemma}\label{Lemma 4.1}
 Let $u$ be the solution of \eqref{1.1-1} satisfying $\|u\|_{\mathbf X_T}^2=\varepsilon^{\frac{4}{3}}$. Then we have
 \begin{equation*}
\sup_{t\in \left[ 0,T\right) }\left\langle t\right\rangle ^{-2\varepsilon
}\left\| u\left( t\right) \right\| _{\mathbf{X}}^{2}+\sup_{t\in \left[
0,T\right) }\left\langle t\right\rangle ^{-\frac{1}{2}+2\varepsilon }\left\|
Ju\left( t\right) \right\| _{\mathbf{L}^{2}}^{2}\leq C\varepsilon ^{2}.
\end{equation*}
\end{lemma}

\begin{proof}
In the same way as in the proof of Lemma \ref{Lemma hom.}, we get by the
energy method 

\begin{equation*}
 \begin{array}{rcl}
\left\| w\left( t\right) \right\| _{\esp^ 2}+\left\| \partial_t w\left( t\right) \right\| _{\esp^ 2}+\left\|\partial_x^2 w\left( t\right) \right\| _{\esp^ 2} &\leq &\displaystyle \|u_0\|_{\mathbf X}^ 2+\displaystyle  C \varepsilon^{\frac{2}{3}}\int \limits_{0}^t\|f(\tau)\|_{\mathbf X}d\tau,
 \end{array}
\end{equation*}
where $f(\tau)=|u|^2u.$ Since
$$\left\| f\left( \tau\right) \right\| _{\mathbf X }\leq C \langle \tau \rangle^{-1+\varepsilon}\|u\|_{\mathbf X_T}^3 $$
we have
\begin{equation}
   \left\| w\left( t\right) \right\| _{\esp^ 2}+\left\| \partial_t w\left( t\right) \right\| _{\esp^ 2}+\left\|\partial_x^2 w\left( t\right) \right\| _{\esp^ 2} \leq \|u_0\|_{\mathbf X}^ 2+C\langle t\rangle^{-2\varepsilon}\|u\|^4_{\mathbf X_T}.\label{4.7}
\end{equation}
From \eqref{2.7} and \eqref{2.15}, we observe that to estimate $\|\partial_x w\|_{\esp^ 2}$ and $\|J w\|_{\esp^ 2}$ we first need to calculate $w(t,0)$. 
In \cite{kaikinaRobin} was proved that
\begin{equation}
 \|J U(t)\phi\|_{\esp^2}\leq C \|\phi\|_{\mathbf{H}^{0,1}}.\label{KR}
\end{equation}
\begin{equation}\label{KR2}
U(t)u_0=\frac{1}{\sqrt{2it}}\Psi\left( \frac{x}{2\tau}\right) +Ct^{-\frac{3}{4}}\|JU(t)u_0\|_{\esp^ 2}
\end{equation} with 
$$\Psi(x,t)=e^{i\frac{x^2}{2t}}\left.\frac{1}{1-i\alpha x/2t}\mathcal{F}_s\mathcal{B}u_0\right|_{p=x/2t}.$$
Since $\lim \limits_{x\to 0}\Psi(x,t)=0$, by virtue of \eqref{t} and \eqref{KR2} we obtain
\begin{equation}
\begin{array}{l}\label{k0}
w(t,0)=\displaystyle t^{-\frac{3}{4}}\|JU(t)u_0\|_{\esp^ 2}+\lim \limits_{x\to 0}\int \limits_0^ t U(t-\tau)|u|^2u(\tau)d\tau
\end{array}
\end{equation}
From the definition of the operator $U$ given in \eqref{UM}
we have
\begin{equation}\label{k1}
\int \limits_{0}^t U(t-\tau)|u|^2u(\tau)d\tau=U(t)\psi, \ \ \psi:=\int\limits_0^ t \mathcal{B}^{-1}\mathcal{F}_s e^{\frac{i}{2} p^2\tau }\mathcal{F}_s\mathcal B |u|^2u(\tau)d\tau,
\end{equation}
newly, applying \eqref{KR2}
we obtain
\begin{equation}
\lim\limits_{x\to 0}\int \limits_{0}^t U(t-\tau)|u|^2u(\tau)d\tau=\lim\limits_{x\to 0}\frac{1}{\sqrt{2it}}e^{i\frac{x^2}{2t}}\left.\frac{1}{1-i\alpha x/2t}\mathcal{F}_s\mathcal{B}\psi\right|_{p=x/2t}+Ct^{-\frac{3}{4}}\|JU(t)\psi\|_{\esp^2}.
\end{equation}
For $\psi$ given by \eqref{k1} we have
$$\lim \limits_{x\to 0}\mathcal{F}_s\mathcal{B}\psi=\lim \limits_{x\to 0}\frac{1}{1-i\alpha x/2t}\int \limits_0^t e^{i(\frac{x^2}{2t})^2\tau}\left(\mathcal{F}_s-\alpha\frac{x}{2t}\mathcal{F}_c\right)|u|^2u(\tau)d\tau=0,$$
therefore
\begin{equation}\label{k2}
\lim\limits_{x\to 0}\int \limits_{0}^t U(t-\tau)|u|^2u(\tau)d\tau=Ct^{-\frac{3}{4}}\|JU(t)\psi\|_{\esp^2}.
\end{equation}
From \eqref{k0}-\eqref{k2} we have
\begin{equation}\label{k3}
w(t,0)=C \displaystyle t^{-\frac{3}{4}}\left(\|JU(t)u_0\|_{\esp^ 2}+\|JU(t)\psi\|_{\esp^2}\right),\  \ \psi:=\int\limits_0^ t \mathcal{B}^{-1}\mathcal{F}_s e^{\frac{i}{2} p^2\tau }\mathcal{F}_s\mathcal B |u|^2u(\tau)d\tau,
\end{equation}
By direct calculation, 
\begin{equation*}
\begin{array}{rcl}
JU(t)\psi&=&\displaystyle \int\limits_0^t \partial_p \frac{1}{1+i\alpha p}e^{i\frac{p^2}{2}\tau} \mathcal{F}_s\mathcal B |u|^2u(\tau)d\tau\\
&=&\displaystyle \int\limits_0^t e^{i\frac{p^2}{2}\tau}\frac{1}{1+i\alpha p}\left[-i\tau |u|^2u(\tau,0)+(\mathcal{F}_s-\alpha p \mathcal{F}_c)\right]J|u|^2u(\tau )d\tau\\
&&\displaystyle +\int\limits_0^t e^{i\frac{p^2}{2}\tau}\left[\partial_p \left( \frac{1}{1+i\alpha p}\right) \mathcal{F}_s |u|^2u(\tau )-\partial_p\left( \frac{\alpha p}{1+i\alpha p}\right)\mathcal{F}_c |u|^2u(\tau ) \right]d\tau,
\end{array}
\end{equation*}
therefore, using the Plhancherel theorem, we obtain
\begin{equation}\label{k4}
\begin{array}{rcl}
\|JU(t)\psi\|_{\esp^ 2}&\leq& \displaystyle C \int\limits_0^t \left( \|J|u|^2u(\tau)\|_{\esp^2}+\||u|^2u(\tau)\|_{\esp^2}+\tau |u(\tau,0)|^3\right)d\tau\\
&\leq & \displaystyle C \left( \langle t \rangle^{\frac{1}{4}-\gamma}+\langle t\rangle^{\gamma}\right) \|u\|_{\mathbf X_T}^3+C \int\limits_0^t\tau | u(\tau,0)|^3d\tau.
\end{array}
\end{equation}
Note that $|u(0,\tau)|\leq \|u\|_{\esp^\infty}$, therefore 
\begin{equation}
 \int\limits_1^t\tau | u(\tau,0)|^3d\tau\leq C\left(\int \limits_0^ t \tau |u(\tau,0)|^ 3+\sup _t t^{\frac{3}{2}}\|u\|^3_{\esp^ 2}\int \limits_0^ 1 \frac{1}{\sqrt{\tau}}d\tau \right)\leq \langle  t \rangle^{\frac{1}{2}}\|u\|_{\mathbf X_T}\label{k5},
\end{equation}
applying the above estimate to \eqref{k3}, we obtain
\begin{equation}\label{4.16}
w(t,0)=C \displaystyle t^{-\frac{3}{4}}\|u_0\|_{\mathbf{H}^{0,1}}+t^{-\frac{1}{4}}\|u\|_{\mathbf X_T}^ 3,
\end{equation}
here we used \eqref{KR}.
Combining \eqref{4.16} with \eqref{2.7}, \eqref{2.15} and \eqref{4.7} we have
\begin{equation}
\sup_{t\in \left[ 0,T\right) }\left\langle t\right\rangle ^{-2\gamma
}\left\| w\left( t\right) \right\| _{\mathbf{X}}^{2}+\sup_{t\in \left[
0,T\right) }\left\langle t\right\rangle ^{-1+2\gamma }\left\|
Jw\left( t\right) \right\| _{\mathbf{L}^{2}
}^{2}\leq C\varepsilon ^{2}.  \label{4.9}
\end{equation}
We next show a-priori estimate of $z\left( t\right) $. We start with (see (%
\ref{2.6}) -(\ref{2.140})) 
\begin{equation}
\left\| z\left( t\right) \right\| _{\mathbf{L}^{2}\left( \mathbf{R}%
^{+}\right) }^{2}\leq \int_{0}^{t}\left|\partial_x z\left( \tau ,0\right) \right|
\left| h\left( \tau \right) \right| d\tau .  
\label{4.0-1}
\end{equation}
\begin{equation}
\left\| \partial _{x}z\left( t\right) \right\| _{\mathbf{L}^{2}\left( 
\mathbf{R}^{+}\right) }^{2}\leq 2\int_{0}^{t}\left| \partial _{\tau }z\left(
\tau ,0\right) \right| \left| \partial_x z(\tau,0)\right| d\tau . \label{4.0-2}
\end{equation}
\begin{equation}
\left\| \partial _{t}z\left( t\right) \right\| _{\mathbf{L}^{2}\left( 
\mathbf{R}^{+}\right) }^{2}\leq C\int_{0}^{t}\left| \partial _{\tau }z\left(
\tau ,0\right) \right| \left| \partial _{\tau }h\left( \tau \right) \right|
d\tau . \label{4.0-3}
\end{equation}
and 
\begin{equation}
\left\Vert Jz\left( t\right) \right\Vert _{\mathbf{L}^{2}\left( \mathbf{R}%
^{+}\right) }^{2}\leq \int_{0}^{t}|\tau z(\tau,0)+2i\tau^{2}\partial _{\tau}z(\tau,0)||\partial _{x}z\left( t,0\right)|.  \label{4.5}
\end{equation}
We have by (\ref{3.14}) and integration by parts 
\begin{eqnarray*}
&&\partial _{x}z\left( t,0\right) \\
&=&C\int_{0}^{t}\frac{1}{\sqrt{\tau }}\partial _{t}h\left( t-\tau \right)
d\tau =\left\{ \int_{0}^{\frac{t}{2}}+\int_{\frac{t}{2}}^{t}\right\} \frac{1%
}{\sqrt{\tau }}\partial _{t}h\left( t-\tau \right) d\tau \\
&=&C\int_{0}^{\frac{t}{2}}\frac{1}{\sqrt{\tau }}\partial _{t}h\left( t-\tau
\right) d\tau -\int_{\frac{t}{2}}^{t}\frac{1}{\sqrt{\tau }}\partial _{\tau
}h\left( t-\tau \right) d\tau \\
&=&C\int_{0}^{\frac{t}{2}}\frac{1}{\sqrt{\tau }}\partial _{t}h\left( t-\tau
\right) d\tau -\int_{\frac{t}{2}}^{t}\partial _{\tau }\left( \frac{1}{\sqrt{%
\tau }}h\left( t-\tau \right) \right) d\tau \\
&&+C\int_{\frac{t}{2}}^{t}\frac{1}{2}\frac{1}{\tau \sqrt{\tau }}h\left(
t-\tau \right) d\tau .
\end{eqnarray*}
By $\left| \partial _{t}h\left( t\right) \right| \leq C\varepsilon
\left\langle t\right\rangle ^{-\frac{7}{4}-\gamma },\left| h\left( t\right)
\right| \leq C\varepsilon \left\langle t\right\rangle ^{-\frac{3}{4}-\gamma
} $ and $h\left( 0\right) =0$ we have 
\begin{equation}
\left| \partial _{x}z\left( t,0\right) \right| \leq C\varepsilon
\left\langle t\right\rangle ^{-\frac{5}{4}-\gamma }
\end{equation}
in a similar way we can obtain
\begin{equation}\label{4.5a}
\left| z\left( t,0\right) \right| \leq C\varepsilon
\left\langle t\right\rangle ^{-\frac{1}{4}-\gamma },\ \ \ \left| \partial _{t}z\left( t,0\right) \right| \leq C\varepsilon
\left\langle t\right\rangle ^{-\frac{5}{4}-\gamma }
\end{equation}
which implies by \eqref{4.0-1}-\eqref{4.0-3}
\begin{eqnarray}
\left\| z\left( t\right) \right\| _{\mathbf{L}^{2}}^2+\left\|\partial_x z\left( t\right) \right\| _{\mathbf{L}^{2}\left( \mathbf{R}^{+}\right) }^{2} +\left\| \partial_t z\left( t\right) \right\| _{\mathbf{L}^{2}}^2
&\leq &C\varepsilon ^{2}\int_{0}^{t}\left\langle \tau \right\rangle
^{-1-\gamma }d\tau \leq C\varepsilon ^{2}  \label{4.11}
\end{eqnarray}
To estimate $\|Jz\|_{\esp^2(\mathbf R^+)}$ we need to analize  
$itz\left( t,0\right) +2it^{2}\partial _{t}z(t,0) $, for this via \eqref{3.13}-\eqref{3.15},
we have by a direct computation 
\begin{eqnarray}
&&itz\left( t,0\right) +2it^{2}\partial _{t}z(t,0)  \notag \\
&=&C\frac{t}{\sqrt{2i\pi }}\int_{0}^{t}\frac{1}{\sqrt{\tau }}h\left( t-\tau
\right) d\tau -2C\frac{t^{2}}{\sqrt{2i\pi }}\int_{0}^{t}\frac{1}{\sqrt{\tau }}%
\partial _{\tau }h\left( t-\tau \right) d\tau  \notag \\
&=&C\frac{t}{\sqrt{2i\pi }}\int_{\frac{t}{2}}^{t}\frac{1}{\sqrt{\tau }}%
h\left( t-\tau \right) d\tau -2C\frac{t^{2}}{\sqrt{2i\pi }}\int_{\frac{t}{2}%
}^{t}\frac{1}{\sqrt{\tau }}\partial _{\tau }h\left( t-\tau \right) d\tau 
\notag \\
&&+C\frac{t}{\sqrt{2i\pi }}\int_{0}^{\frac{t}{2}}\frac{1}{\sqrt{\tau }}%
h\left( t-\tau \right) d\tau -2C\frac{t^{2}}{\sqrt{2i\pi }}\int_{0}^{\frac{t}{%
2}}\frac{1}{\sqrt{\tau }}\partial _{\tau }h\left( t-\tau \right) d\tau .
\label{4.6}
\end{eqnarray}
We apply the integration by parts to the first and second terms of the right
hand side to find that 
\begin{eqnarray*}
&&\frac{t}{\sqrt{2i\pi }}\int_{\frac{t}{2}}^{t}\frac{1}{\sqrt{\tau }}h\left(
t-\tau \right) d\tau -2\frac{t^{2}}{\sqrt{2i\pi }}\int_{\frac{t}{2}}^{t}%
\frac{1}{\sqrt{\tau }}\partial _{\tau }h\left( t-\tau \right) d\tau \\
&=&-2\frac{t^{2}}{\sqrt{2i\pi }}\int_{\frac{t}{2}}^{t}\partial _{\tau
}\left( \frac{1}{\sqrt{\tau }}h\left( t-\tau \right) \right) d\tau \\
&&+\frac{t}{\sqrt{2i\pi }}\int_{\frac{t}{2}}^{t}\frac{1}{\sqrt{\tau }}%
h\left( t-\tau \right) d\tau -\frac{t^{2}}{\sqrt{2i\pi }}\int_{\frac{t}{2}%
}^{t}\frac{1}{\tau \sqrt{\tau }}h\left( t-\tau \right) d\tau \\
&=&\frac{2t^{\frac{3}{2}}}{\sqrt{i\pi }}h\left( \frac{t}{2}\right) -\frac{%
t^{2}}{\sqrt{2i\pi }}\int_{\frac{t}{2}}^{t}\left( \frac{t-\tau }{t\tau }%
\right) \frac{1}{\sqrt{\tau }}h\left( t-\tau \right) d\tau .
\end{eqnarray*}
Hence 
\begin{eqnarray}
\left\vert \frac{t}{\sqrt{2i\pi }}\int_{\frac{t}{2}}^{t}\frac{1}{\sqrt{%
\tau }}h\left( t-\tau \right) d\tau\right. & -&\left.2\frac{t^{2}}{\sqrt{2i\pi }}\int_{\frac{%
t}{2}}^{t}\frac{1}{\sqrt{\tau }}\partial _{\tau }h\left( t-\tau \right)
d\tau \right\vert  \notag \\
&\leq &C\varepsilon\left\langle t\right\rangle
^{\frac{3}{4}-\gamma }+C\varepsilon \int_{\frac{t}{2}}^{t}\frac{1}{\sqrt{\tau }%
\left\langle t-\tau \right\rangle ^{\gamma }}d\tau  \notag \\
&\leq &C\varepsilon\left\langle t\right\rangle 
^{\frac{3}{4}-\gamma } \label{4.7 1}
\end{eqnarray}
The third and fourth terms of the right hand sides of (\ref{4.6}) are
estimated from above by 
\begin{eqnarray}
&&C\varepsilon ^{3}t\int_{0}^{\frac{t}{2}}\frac{1}{\sqrt{\tau }\left\langle
t-\tau \right\rangle ^{1+\gamma }}d\tau +C\varepsilon ^{3}t^{2}\int_{0}^{%
\frac{t}{2}}\frac{1}{\sqrt{\tau }\left\langle t-\tau \right\rangle
^{2+\gamma }}d\tau  \notag \\
&\leq &C\varepsilon \left\langle t\right\rangle
^{\frac{3}{4}-\gamma },  \label{4.7-1}
\end{eqnarray}
where we have used the assumption such that $\left\vert \partial _{t}h\left(
t\right) \right\vert \leq C\varepsilon \left\langle t\right\rangle
^{-\frac{7}{4}-\gamma }.$ We apply (\ref{4.7 1}) and (\ref{4.7-1}) to (\ref{4.6}) to
have 
\begin{equation*}
\left\vert tz\left( t,0\right) +2t^{2}\partial _{t}z(t,0)\right\vert \leq
C\varepsilon\left\langle t\right\rangle ^{\frac{3}{4}-\gamma }
\end{equation*}
for $t\geq 1$, which gives us 
\begin{eqnarray*}
\left\Vert Jz\left( t\right) \right\Vert _{\mathbf{L}^{2}\left( \mathbf{R}%
^{+}\right) }^{2} &\leq &C\varepsilon ^{2}\int_{0}^{t}\left\langle \tau
\right\rangle ^{-\frac{5}{4}-\gamma }\langle \tau
\rangle ^{\frac{3}{4}-\gamma }d\tau  \leq
C\varepsilon ^{2}\left\langle t\right\rangle ^{\frac{1}{2}-2\gamma }
\end{eqnarray*}
for $\gamma \geq 0.$
This gives us the desired estimate
\begin{equation}\label{4.10}
\left\langle t\right\rangle ^{-\frac{1}{2}+2\gamma }\left\Vert Jz\left( t\right) \right\Vert _{\mathbf{L}^{2}\left( \mathbf{R}%
^{+}\right) }^{2}\leq C\epsilon^2.
\end{equation}
From (\ref{4.11})-(\ref{4.10}), the estimate 
\begin{equation}
\sup_{t\in \left[ 0,T\right) }\left\| z\left( t\right) \right\| _{\mathbf{X}%
}^{2}+\sup_{t\in \left[ 0,T\right) }\left\langle t\right\rangle ^{-\frac{1}{2%
}+2\gamma }\left\| Jz\left( t\right) \right\| _{\mathbf{L}^{2}\left( \mathbf{%
R}^{+}\right) }^{2}\leq C\varepsilon ^{2}  \label{4.12}.
\end{equation}
follows. Since $u\left( t\right) =w\left( t\right) +z\left( t\right) ,$ by (
\ref{4.9}) and (\ref{4.12}) 
\begin{equation*}
\sup_{t\in \left[ 0,T\right) }\left\langle t\right\rangle ^{-2\varepsilon
}\left\| u\left( t\right) \right\| _{\mathbf{X}}^{2}+\sup_{t\in \left[
0,T\right) }\left\langle t\right\rangle ^{-\frac{1}{2}+2\varepsilon }\left\|
Ju\left( t\right) \right\| _{\mathbf{L}^{2}
}^{2}\leq C\varepsilon ^{2}.
\end{equation*}
This completes the proof of the Lemma.
\end{proof}

In order to prove the a-priori estimate of solutions in the uniform norm we use the factorization technique of the evolution operator used in  \cite{EsHaKa}. We note
\begin{equation}
U(t)\psi=\mathcal B^{-1}U_D(t) \mathcal B \psi=\mathcal B^{-1}MD_t\mathcal{F}_s M \mathcal B \psi
\end{equation}
where $M=e^{i\frac{|x|^2}{2t}}$, $D_t\phi=\frac{1}{\sqrt{it}}\phi(\frac{x}{\sqrt{t}})$.
By a direct calculation we have
\begin{equation*}
U\left( t\right) ^{-1}=\mathcal B^{-1}U_D^{-1}(t) \mathcal B =\mathcal B^{-1}U_D(-t) \mathcal B =U\left(
-t\right) .
\end{equation*}
We have
\begin{eqnarray*}
\psi \left( t\right) &=&U\left( t\right) U\left( -t\right) \psi
\left( t\right) =\mathcal B^{-1}MD_t\mathcal{F}_s M \mathcal B U\left( -t\right) \psi \left(
t\right) \\
&=&MD_t\mathcal{F}_s  \mathcal B U\left( -t\right) \psi \left( t\right) +\mathcal B^{-1}MD_t\mathcal{F}_s (M -1)\mathcal B  U\left( -t\right) \psi \left( t\right)\\
&&+(\mathcal B^{-1}-1)MD_t\mathcal{F}_s  \mathcal B U\left( -t\right) \psi \left( t\right) 
\end{eqnarray*}
which implies via Lemma \ref{Lemma 1.3} 
\begin{equation}
\begin{array}{l}
 \left\| \psi \left( t\right) \right\| _{\mathbf{L}^{\infty }\left( \mathbf{R}%
^{+}\right) }\leq Ct^{-\frac{1}{2}}\left\|\mathcal{F}_s \mathcal B U(-t) \psi \left( t\right) \right\| _{\mathbf{L}^{\infty }\left( \mathbf{%
R}^{+}\right) }\\
\hspace{2cm}+Ct^{-\frac{3}{4}}\left\| xU\left( -t\right) \psi \left(
t\right) \right\| _{\mathbf{L}^{2}}+\|\partial_xMD_t\mathcal{F}_s  \mathcal B U\left( -t\right) \psi \|_{\esp^2}
\end{array}
\label{4.0-b}
\end{equation}

Via integration by parts, if $\mathcal{B}\psi(0)=0$ we have
\begin{align*}
 xU(-t)\psi=&\int \limits_{-\infty}^\infty \partial_p(e^{ipx})e^{-i\frac{p^2}{2}t}\frac{1}{1+i\alpha p}\mathcal{F}_s\mathcal B\psi(p) dp\\
= &\int \limits_{-\infty}^\infty e^{ipx}e^{-i\frac{p^2}{2}t}\frac{1}{1+i\alpha p}(ipt-\partial_p)\mathcal{F}_s\mathcal B\psi(p)  dp\\
 &+\int \limits_{-\infty}^\infty e^{ipx}e^{-i\frac{p^2}{2}t}\frac{i\alpha }{(1+i\alpha p)^2}\mathcal{F}_s\mathcal B\psi(p)  dp\\
 =&\int \limits_{-\infty}^\infty e^{ipx}e^{-i\frac{p^2}{2}t}\frac{1}{1+i\alpha p}\F_c (\mathcal{J }+1) \psi dp\\
 &+\int \limits_{-\infty}^\infty e^{ipx}e^{-i\frac{p^2}{2}t}\frac{i\alpha }{(1+i\alpha p)^2}\F_s\mathcal{B}\psi(p) dp\\
\end{align*}
Hence, via Plancherel theorem, we have 
\begin{eqnarray}
\label{rf6}\\
\left\| xU\left( -t\right) \psi\right\| _{\mathbf{L}^{2}} &\leq &\left\| \frac{1}{1+i\alpha p}\F_c( \mathcal{J}+1) \psi\right\| _{%
\mathbf{L}^{2} }+\left\|\frac{i\alpha }{(1+i\alpha p)^2}\left\{ \mathcal{F}_s-i\alpha p \mathcal F_c\right\}\psi(p) \right\|_{\esp^ 2} \notag \\
&\leq&C\left(\left\| \mathcal J\psi\right\| _{\mathbf{L}^{2} }+\|\psi\|_{\esp^2}\right).\notag
\end{eqnarray}
In the same way we can prove
\begin{equation}\label{rf60}
 \left\| x\partial_x U\left( -t\right) \psi\right\| _{\mathbf{L}^{2}}\leq C \left(\left\| \mathcal J\psi\right\| _{\mathbf{L}^{2}}+\|\psi\|_{\esp^2}+\|\partial_x \psi\|_{\esp^ 2}\right).\
\end{equation}
On the other hand, we have
\begin{equation}
 \begin{array}{rcl}
  \partial_xMD_t\mathcal{F}_s  \mathcal B U\left( -t\right) \psi=t^{-1}MD_t\mathcal F_cMD_t\mathcal F_c(p\mathcal B\psi)
 \end{array}
\end{equation}
hence
\begin{equation}
 \| \partial_xMD_t\mathcal{F}_s  \mathcal B U\left( -t\right) \psi\|_{\esp^2}\leq Ct^{-2}\|\psi\|_{\mathbf H^{1,1}}
\end{equation}
Therefore by (\ref{4.0-b}) 
\begin{equation}
\left\| \psi\left( t\right) \right\| _{\mathbf{L}^{\infty }\left( \mathbf{R}%
^{+}\right) }\leq Ct^{-\frac{1}{2}}\left\| \mathcal{F}_{s}\mathcal B U\left( -t\right)
\psi\right\| _{\mathbf{L}^{\infty }}+Ct^{-\frac{3}{%
4}}\left(\left\| \mathcal J\psi\right\| _{\mathbf{L}^{2}}+\left\| \psi\right\| _{\mathbf{L}^{2}}\right)+Ct^{-2}\|\psi\|_{\mathbf H^{1,1}}.
\label{4.0-5}
\end{equation}

\begin{lemma}
\label{Lemma 4.2}Let $u$ be the solution of (\ref{1.1-1}) . Then we have 
\begin{equation*}
\left\| u\left( t\right) \right\| _{\mathbf{L}^{\infty }\left( \mathbf{R}%
^{+}\right) }\left\langle t\right\rangle ^{\frac{1}{2}}\leq C\varepsilon .
\end{equation*}
\end{lemma}

\begin{proof}
We have by a direct calculation 
\begin{eqnarray*}
z=\frac{2}{\sqrt{2\pi i}}\mathcal{B}^{-1}\int \limits_{0}^{t}h(t-\tau)\frac{x}{t\sqrt{\tau}}e^{i\frac{x^2}{2\tau}}d\tau=\mathcal{B}^{-1}z_{D}
\end{eqnarray*}
where
$z_D$ is the  solution of the initial boundary problem
\begin{equation}
\left\{ 
\begin{array}{c}
Lz_D=0,x\mathbf{\in R}^{+},\text{ }t>0, \\ 
z_D(0,x)=0,\text{ \ }x\text{ \ }\mathbf{\in R}^{+}, \\ 
z_D\left( t,0\right)=h\left( t\right) .%
\end{array}
\right.  
\end{equation}
Therefore
\begin{equation}
 \F_s\mathcal B U(-t)z=\F_sU_D(-t)z_D,\label{zd}
\end{equation}
In \cite{EsHaKa}, Lemma 4,2, we proved 
$$|\F_sU_D(-t)z_D|\leq C\varepsilon,$$ if $|h(t)|\leq C\varepsilon^2\langle t \rangle^{-\frac{1}{2}-\gamma}$, as consequence
\begin{equation}
 \|\F_s\mathcal B U(-t)z\|_{\esp^\infty}\leq C\varepsilon. \label{4}
\end{equation}
Now we estimate $\|z\|_{\mathbf H^{1,1}}$, for this via \eqref{rf6}, \eqref{rf60}  and \eqref{KR} we observe
\begin{equation}
\begin{array}{rcl}\label{4.40}
\|z\|_{\mathbf H^{1,1}}&=&\|U(-t)U(t)z\|_{\mathbf H^{1,1}}\leq C\left( \|U(t)z\|_{\mathbf H^1}+\|JU(t)z\|_{\esp^ 2} \right)\\
&\leq& C  \left( \|z\|_{\mathbf H^1}+\|xz\|_{\esp^ 2} \right),
\end{array}
\end{equation}
as consequence, we need estimate $\|xz\|_{\esp^ 2}.$
From \eqref{bfz} we have
\begin{equation}
z(t,x)=\int\limits_0^t I(x,\tau)h(t-\tau)d\tau, \ \ I(\tau,x)=\int\limits_{\mathbb R}e^{ipx} e^{i\frac{p^2}{2}\tau}\frac{p}{1+i\alpha p} dp.
\end{equation}
Let us define $I_b$ as
\begin{equation}
I_b(\tau,x)=\int\limits_{\mathbb R}e^{ipx} e^{-b|p|}e^{i\frac{p^2}{2}\tau}\frac{p}{1+i\alpha p} dp,  \ \ b>0,
\end{equation}
we note $$I(\tau,x)=\lim\limits_{b\to 0}I_b(\tau,x).$$ On the other hand, via integration by parts we observe
\begin{equation}
\begin{array}{rcl}
xI(\tau,x)&=&\lim\limits_{b\to 0}xI_b(\tau,x)\\
&=&\lim\limits_{b\to 0} \displaystyle -i\int\limits_{\mathbb R}\partial_p(e^{ipx}) e^{-b|p|}e^{i\frac{p^2}{2}\tau}\frac{p}{1+i\alpha p}dp\\
&=&\displaystyle  i\int\limits_{\mathbb R}e^{ipx}e^{\frac{i\tau p^2}{2}} \left[ 
\dfrac{1}{(1+i\alpha p)^2}+\dfrac{i\tau p^2}{1+i\alpha p} \right]
dp\end{array}
\end{equation}
this implies
$$
\begin{array}{rcl}
xz(t,x)=\displaystyle i\int\limits_0^t h(t-\tau)\int\limits_{\mathbb R}e^{ipx}e^{\frac{i\tau p^2}{2}} 
\dfrac{1}{(1+i\alpha p)^2}
dp\ d\tau\\\hspace{2cm}+i\displaystyle \int\limits_0^t \tau h(t-\tau)\int\limits_{\mathbb R}e^{ipx}\partial_\tau(e^{\frac{i\tau p^2}{2}}) \dfrac{1}{1+i\alpha p} dp \ d\tau
\end{array} $$
integrating by parts we get
\begin{equation}
\begin{array}{rcl}\label{r5}
xz(x,t)&=&\displaystyle i\int\limits_0^t h(t-\tau)\int\limits_{\mathbb R}e^{ipx}e^{\frac{i\tau p^2}{2}} 
\dfrac{1}{(1+i\alpha p)^2}
dp\ d\tau\\
&&\displaystyle +i\int\limits_0^t \left(h(t-\tau)-\tau\partial_\tau h(t-\tau)\right)\int\limits_{\mathbb R}e^{ipx}e^{\frac{i\tau p^2}{2}} \dfrac{1}{1+i\alpha p} dp \ d\tau.
\end{array}
\end{equation}
since $\left| h\left( t\right) \right|+\langle t\rangle\left| \partial_\tau h\left( t\right) \right|+ \langle t\rangle^ 2\left| \partial_\tau h\left( t\right) \right|\leq C\varepsilon \left\langle t\right\rangle ^{-\frac{3}{4}-\gamma },$ we obtain $\|xz\|_{\esp^ 2}\leq C\epsilon \langle t\rangle^{\frac{1}{4}-\gamma}$. From \eqref{4.40}
\begin{equation}
\|z\|_{\mathbf{H}^{1,1}} \leq C\epsilon \langle t\rangle^{\frac{1}{4}-\gamma}\label{4.20.0}
\end{equation}
By (\ref{4.0-5}), (\ref{4}) and (\ref{2}) 
\begin{equation}
\left\| z\left( t\right) \right\| _{\mathbf{L}^{\infty }\left( \mathbf{R}%
^{+}\right) }\leq C\varepsilon t^{-\frac{1}{2}}+Ct^{-\frac{3}{4}}\left\|
Jz\left( t\right) \right\| _{\mathbf{L}^{2}
}+Ct^{-\frac{3}{2}} \|z\|_{\mathbf H^{1,1}}\leq C\varepsilon t^{-\frac{1}{2}} \label{4.20.a}
\end{equation}
if $\left| h\left( t\right) \right|+\langle t\rangle\left| \partial_\tau h\left( t\right) \right|+ \langle t\rangle^ 2\left| \partial_\tau h\left( t\right) \right|\leq C\varepsilon \left\langle t\right\rangle ^{-\frac{3}{4}-\gamma }$.


We next estimate $\left\Vert w\left( t\right) \right\Vert _{\mathbf{L}%
^{\infty } }\left\langle t\right\rangle ^{\frac{1%
}{2}}.$ For this we write 
\begin{equation*}
u=y+e^{-x}h\left( t\right) ,
\end{equation*}%
where $\mathcal{B} y\left( t,0\right) =0.$ \ \ We represent the nonlinear term as 
\begin{eqnarray}\label{nonlinear term}
\left\vert u\right\vert ^{2}u &=&\left\vert y+e^{-x}h\left( t\right)
\right\vert ^{2}\left( y+e^{-x}h\left( t\right) \right) \\
&=&\left\vert y\right\vert ^{2}y+R,
\end{eqnarray}%
where 
\begin{eqnarray*}
R &=&e^{-x}y^2\overline{h}(t)+2e^{-x}|y|^2h+2ye^{-2x}|h|^2+e^{-2x}h^2\overline{y}+e^{-3x}|h|^2h
\end{eqnarray*}%
Applying $\mathcal{F}_s \mathcal B U(-t) $ to both sides \eqref{4.3}
 we obtain by using the factorization formula of the evolution operator 
 \begin{equation}\label{H}
   \begin{array}{lcr}
  i\partial _{t}\mathcal{F}_s \mathcal B U(-t)  \left( y+e^{-x}h\left(
t\right) \right) \\
=\lambda \mathcal{F}_s  U_D(-t) \mathcal B\left| \mathcal B^{-1} U_D(t) \mathcal B U\left( -t\right) y \right|^ 2 \mathcal B^{-1} U_D(t) \mathcal B U\left( -t\right) y+ \lambda \mathcal{F}_s  U_D(-t) \mathcal B R+G\\
=\lambda \mathcal{F}_s  U_D(-t)\left|\tilde{y}\right|^ 2\tilde{y}+\lambda \mathcal{F}_s  U_D(-t) R_1+ \lambda \mathcal{F}_s \mathcal B U(-t)  R+G
\end{array}
 \end{equation}
where
$$\tilde{y}=U_D(t) \mathcal B U\left( -t\right) y $$
$$\begin{array}{lcr}
   R_1=\mathcal{B}(|\mathcal{B}^{-1}\tilde{y}|^2\mathcal{B}^{-1}\tilde{y})-3|\tilde y|^2\tilde y
  \end{array}
$$
$$G=i\partial _{t}\mathcal{F}_s \mathcal B U(-t)  z$$
 Since $U_D(t)=MD_{t}\F_sM$ and $U(-t)=B^{-1}U_{D}(-t)B$ we have
  \begin{equation}
   \begin{array}{lcr}
  i\partial _{t}\mathcal{F}_s \mathcal B  U(-t) \left( y+e^{-x}h\left(
t\right) \right)\\
=\lambda\mathcal{F}_s  U_D(-t) \left|MD_{t}\F_sM \mathcal B U\left( -t\right) y\right|^ 2  MD_{t}\F_sM \mathcal B U\left( -t\right) y\\
\hspace{7cm}+R_1+ \lambda \mathcal{F}_s \mathcal B U_D(-t)  R+G\\
=t^{-1}\lambda\mathcal{F}_s  M^{-1}\F_s ^{-1}\left|\F_sM \mathcal B U\left( -t\right) y \right|^ 2  \F_sM \mathcal B U\left( -t\right) y+R_1+ \lambda \mathcal{F}_s \mathcal B  U_D(-t)  R+G\\
=t^{-1}\lambda\mathcal{F}_s  M^{-1}\F_s ^{-1}\left|\F_sM \mathcal B U\left( -t\right) y \right|^ 2  \F_sM \mathcal B U\left( -t\right) y\\
\hspace{7cm}+R_1+R_2+R_3+ \lambda \mathcal{F}_s \mathcal B  U_D(-t) R+G.
\end{array}
 \end{equation}
 where
$$
 R_2=3\lambda t^{-1}\lambda\mathcal{F}_s ( M^{-1}-1)\F_s^{-1} \left|\F_sM \mathcal B U\left( -t\right) y \right|^ 2  \F_sM \mathcal B U\left( -t\right) y
$$
$$R_3=3\lambda t^{-1}\left(\left|\F_sM \mathcal B U\left( -t\right) y \right|^ 2  \F_sM \mathcal B U\left( -t\right) y-\left|\F_s \mathcal B U\left( -t\right) y \right|^ 2  \F_s\mathcal B U\left( -t\right) y\right)$$
We note,
\begin{equation}
\|R_2(t)\|_{\esp^\infty}\leq Ct^{-\frac{5}{4}}\|x\mathcal{F}_s^{-1}|\Psi|^2\Psi\|_{\esp^2}\label{3r}
\end{equation}
with $\Psi=\mathcal{F}_s M \mathcal B U\left( -t\right) y $, by directly calculation we have
\begin{eqnarray}
\left\| x\mathcal{F}_{s}^{-1}\left|\Psi \right| ^{2}\Psi
\right\| _{\mathbf{L}^{2}} &\leq &C\left\| \partial
_{x}\left| \Psi \right| ^{2}\Psi \right\| _{\mathbf{L}^{2}}
\label{4r} \\
&\leq &\left\| \Psi \right\| _{\mathbf{L}^{\infty }}^{2}\left\| \partial _{x}\Psi \right\| _{\mathbf{L}^{2}}.  \notag
\end{eqnarray}
Also we get 
\begin{align}
\left\| \partial _{x}\Psi \right\| _{\mathbf{L}^{2}}&=\left\|
\partial _{x}\mathcal{F}_{c}MU\left( -t\right) y\right\| _{\mathbf{L}^{2}}\\
&\leq C\left\| xU\left( -t\right) y\right\| _{\mathbf{L}^{2}}\leq C\left(\left\| Jy\right\| _{\mathbf{L}^{2} }+\|y\|_{\esp^2}\right).  \label{5r}
\end{align}
Therefore via (\ref{4r}) and (\ref{5r}) 
\begin{equation}
\left\| x\mathcal{F}_{s}^{-1}\left| \Psi \right| ^{2}\Psi
\right\| _{\mathbf{L}^{2}}\leq \left\| \Psi \right\| _{%
\mathbf{L}^{\infty }}^{2}\left(\left\| Jy\right\| _{\mathbf{L}^{2} }+\|y\|_{\esp^2}\right).  \label{4.001}
\end{equation}
since we have by Lemma \ref{Lemma 4.1}
\begin{equation*}
\left(\left\| Jy\right\| _{\mathbf{L}^{2} }+\|y\|_{\esp^2}\right)\leq C\varepsilon t^{\frac{1}{4}-\varepsilon }
\end{equation*}
if $\left| h\left( t\right) \right| \leq C\varepsilon ^{3}t^{-1-\varepsilon
}.$ 
Thus by \eqref{3r} we obtain
\begin{equation} \label{C}
 \|R_2(t)\|_{\esp^\infty}\leq Ct^{-\frac{5}{4}}\varepsilon^3.
\end{equation}

Similarly we get 
\begin{eqnarray}
\left\| R_{3}\left( t\right) \right\| _{\mathbf{L}^{\infty }(R^{+})} &\leq
&Ct^{-1}\left( \left\| \mathcal{F}_{s}MBU\left( -t\right) u\right\| _{\mathbf{%
L}^{\infty }(\mathbf{R}^{+})}^{2}+\left\| \mathcal{F}_s\mathcal B U\left( -t\right)
u\right\| _{\mathbf{L}^{\infty }(\mathbf{R}^{+})}^{2}\right)  \notag \\
&&\times \left\| \mathcal{F}_{s}\left( M-1\right) BU\left( -t\right)
u\right\| _{\mathbf{L}^{\infty }(\mathbf{R}^{+})}  \notag \\
&\leq &Ct^{-1-\frac{1}{4}}\left\| \mathcal{F}_s\mathcal B U\left( -t\right) u\right\|
_{\mathbf{L}^{\infty }(\mathbf{R}^{+})}^{2}\left(\left\| Ju\right\| _{\mathbf{L}%
^{2}}+\|u\|_{\esp^2}\right)  \notag \\
&\leq &C\varepsilon ^{3}t^{-1-\varepsilon }.  \label{D}
\end{eqnarray}
if $\left| h\left( t\right) \right| \leq C\varepsilon ^{3}t^{-1-\varepsilon
}.$

Furthermore we have 
\begin{eqnarray*}
&&\left\Vert \mathcal{F}_{s} \mathcal BU\left( -t\right) R\right\Vert _{\mathbf{L}%
^{\infty }(\mathbf{R}^{+})}\\
& \notag \leq &C\left\Vert \mathcal B R\right\Vert _{\mathbf{L}^{1}(\mathbf{R}^{+})}  \notag \\
&\leq &C\left(\left\Vert y\right\Vert _{\mathbf{L}^{\infty} (\mathbf{R}^{+})}^2+\left\Vert y\right\Vert _{\mathbf{L}^{\infty} (\mathbf{R}^{+})}\|\partial_x y\|_{\mathbf{L}^{2 }}\right)\left\vert h\left( t\right) \right\vert \notag\\ 
&& \hspace{3cm}+C\left(\left\Vert y\right\Vert _{\mathbf{L}^{2 }}+\|\partial_x y\|_{\mathbf{L}^{2 }}\right)\left\vert h\left( t\right) \right\vert
^{2}+C\left\vert h\left( t\right) \right\vert ^{3} , \notag 
\end{eqnarray*}
By Sobolev embebeding theorem we have
\begin{eqnarray}
 \label{G2}
&&\left\Vert \mathcal{F}_{s} \mathcal BU\left( -t\right) R\right\Vert _{\mathbf{L}%
^{\infty }(\mathbf{R}^{+})}
\\
&\leq &C\left(\left\Vert u\right\Vert _{\mathbf{X}}^{2}+\|u\|_{\esp^\infty}\| u\|_{\mathbf X}\right)\left\vert h\left( t\right) \right\vert+C\|u\|_{\mathbf{X}}|h(t)|^2
+C\left\vert h\left( t\right) \right\vert ^{3}  \notag \\
&\leq &C\varepsilon ^{3}t^{-1-\varepsilon }  \notag 
\end{eqnarray}
because $\left\vert h\left( t\right) \right\vert \leq C\varepsilon ^{3}t^{-1-\varepsilon }.$
On the other hand, since $\tilde{y}=U_D(t) \mathcal B U\left( -t\right) y $ 
we have
\begin{align}
 \left\Vert \mathcal{F}_{s} U_D\left( -t\right) R_1\right\Vert _{\mathbf{L}%
^{\infty }(\mathbf{R}^{+})}\leq& \|R_1\|_{\esp^1 }\notag
\leq C\|\partial_x\tilde y\|_{\esp^2}^2\|\tilde y\|_{\esp^\infty}\\
\leq & C
t^{-1}\|\partial_xU_D\left( -t\right)\mathcal B y\|_{\esp^2}^2\|\mathcal{B}y\|_{\esp^\infty}\\
\leq & C t^{-2}\|\partial_x \mathcal{B}y\|^{\frac{5}{2}}_{\esp^2}\|\mathcal{B}y\|^{\frac{1}{2}}_{\esp^2}\\
\leq& C  t^{-1-\varepsilon} \varepsilon^3
\end{align}

 Thus applying (\ref{C})-(\ref{G2}) into (\ref{H}) we get 
\begin{equation*}
i\partial _{t}\mathcal{F}_{s}\mathcal BU\left( -t\right) \left( y+e^{-x}h\left(
t\right) \right) =\lambda t^{-1}\left\vert \mathcal{F}_s\mathcal B U\left(
-t\right) y\right\vert ^{2}\mathcal{F}_s\mathcal B U\left( -t\right) y+O\left(
\varepsilon ^{3}t^{-1-\varepsilon }\right) .\label{D1}
\end{equation*}%

We rewrite this identity as 
\begin{eqnarray}
&&i\partial _{t}\mathcal{F}_{s}\mathcal BU\left( -t\right) \left( y+e^{-x}h\left(
t\right) \right)  \notag \\
&=&\lambda t^{-1}\left\vert \mathcal{F}_s\mathcal B U\left( -t\right)
y\right\vert ^{2} \mathcal{F}_s\mathcal B U\left( -t\right)\left( y+e^{-x}h\left(
t\right) \right)  \notag \\
&&-\lambda t^{-1}\left\vert \mathcal{F}_s\mathcal B U\left( -t\right)
y\right\vert ^{2}e^{-x}h\left( t\right) +O\left( \varepsilon
^{3}t^{-1-\varepsilon }\right)  \notag \\
&=&\lambda t^{-1}\left\vert \mathcal{F}_s\mathcal B U\left( -t\right)
y\right\vert ^{2}\mathcal{F}_s\mathcal B U\left( -t\right) \left( y+e^{-x}h\left(
t\right) \right) +O\left( \varepsilon ^{3}t^{-1-\varepsilon }\right)
\label{4.21}
\end{eqnarray}%
here we used
\begin{eqnarray*}
&&t^{-1}\left\vert \left\vert \mathcal{F}_{s}\mathcal B U\left( -t\right)
y\right\vert ^{2} \mathcal{F}_{s}\mathcal B U\left( -t\right)e^{-x}h\left( t\right) \right\vert \\
&\leq &Ct^{-1}\left( \left\Vert \mathcal{F}_{s}U_{D}\left( -t\right)
z\right\Vert _{\mathbf{L}^{\infty }(\mathbf{R}^{+})}^{2}+\left\Vert \mathcal{%
F}_{s}U_{D}\left( -t\right) w\right\Vert _{\mathbf{L}^{\infty }(\mathbf{R}%
^{+})}^{2}\right) \left\vert h\left( t\right) \right\vert \\
&\leq &Ct^{-1}\varepsilon ^{2}\left( 1+t^{\frac{1}{4}-\varepsilon }\right)
\left\vert h\left( t\right) \right\vert .
\end{eqnarray*}%
Multiplying both sides of (\ref{4.21}) by 
\begin{equation*}
e^{-i\lambda \int_{1}^{t}\left\vert \mathcal{F}_{s}\mathcal B U\left( -\tau \right)
y\right\vert ^{2}d\tau }
\end{equation*}%
we obtain 
\begin{eqnarray*}
&&i\partial _{t}\left( e^{-i\lambda \int_{1}^{t}\left\vert \mathcal{F}%
_{s}\mathcal BU\left( -\tau \right) y\right\vert ^{2}d\tau }\mathcal{F}%
_{s}\mathcal BU\left( -t\right) \left( y+e^{-x}h\left( t\right) \right) \right) \\
&=&e^{-i\lambda \int_{1}^{t}\left\vert \mathcal{F}_{s}\mathcal B U\left( -\tau
\right) y\right\vert ^{2}d\tau }\left( G+O\left( \varepsilon ^{3}t^{-1-\varepsilon
}\right) \right).
\end{eqnarray*}%
Integrating in time, we obtain 
\begin{eqnarray*}
&&e^{-i\lambda \int_{1}^{t}\tau ^{-1}\left| \mathcal{F}_{s}\mathcal BU\left( -\tau
\right) u\right| ^{2}d\tau }\mathcal{F}\mathcal BU\left( -t\right) u \\
&=&\mathcal{F}\mathcal BU\left( -1\right) u\left( 1\right)
+\int_{1}^{t}e^{-i\lambda \int_{1}^{t}\tau ^{-1}\left| \mathcal{F}%
_{s}\mathcal BU\left( -\tau \right) u\right| ^{2}d\tau }\left( G+O\left( \varepsilon ^{3}t^{-1-\varepsilon
}\right) \right) dt
\end{eqnarray*}
By \eqref{zd} we have
$$G=i\partial_t \mathcal{F}_{s}U_D\left( -t\right) z_D(t,x)=-\xi e^{-i\frac{\xi}{2}t}h (t)$$
hence 
\begin{align}
 &\int_{1}^{t}e^{-i\lambda \int_{1}^{t}\tau ^{-1}\left| \mathcal{F}%
_{s}\mathcal BU\left( -\tau \right) u\right| ^{2}d\tau }\left( G+O\left( \varepsilon ^{3}t^{-1-\varepsilon
}\right) \right) dt\\
&\hspace{1cm}=\int_{1}^{t}e^{-i\lambda \int_{1}^{t}\tau ^{-1}\left| \mathcal{F}%
_{s}\mathcal BU\left( -\tau \right) u\right| ^{2}d\tau }\left( -\xi e^{i\frac{\xi}{2}t}h (t)+O\left( \varepsilon ^{3}t^{-1-\varepsilon
}\right) \right) dt \notag\\
&\hspace{1cm}=C\varepsilon^3-\int_{1}^{t}e^{-i\lambda \int_{1}^{t}\tau ^{-1}\left| \mathcal{F}%
_{s}\mathcal BU\left( -\tau \right) u\right| ^{2}d\tau }\xi e^{-i\frac{\xi}{2}t}h (t)dt \notag
\end{align}
Via a similar process, which was shown by us in \cite{EsHaKa}, we can prove
\begin{align}
 &\left|\int_{1}^{t}e^{-i\lambda \int_{1}^{t}\tau ^{-1}\left| \mathcal{F}%
_{s}\mathcal BU\left( -\tau \right) u\right| ^{2}d\tau }\xi e^{-i\frac{\xi}{2}t}h (t)dt\right|
\leq C\varepsilon\notag
\end{align}
Collecting everything, we arrive at
\begin{equation*}
\left\| \mathcal{F}_{s}\mathcal BU\left( -t\right) (y+e^{-x}h(t))\right\| _{\mathbf{L}^{\infty }(%
\mathbf{R}^{+})}\leq C\varepsilon .
\end{equation*}
Therefore 
\begin{equation}
\left\| \mathcal{F}_{s}\mathcal B U\left( -t\right) w\right\| _{\mathbf{L}^{\infty }(%
\mathbf{R}^{+})}\leq C\varepsilon .  \label{4.22}
\end{equation}
hence
\begin{equation}
\left\| w\right\| _{\mathbf{L}^{\infty }}\leq C\varepsilon
t^{-\frac{1}{2}}\left\| \mathcal{F}_{c}U\left( -t\right) w\right\| _{\mathbf{%
L}^{\infty }}+C\varepsilon t^{-\frac{3}{4}}\left(\left\|
Jw\right\| _{\mathbf{L}^{2}}+\|w\|_{\esp^2}\right)+Ct^{-\frac{3}{2}}\|w\|_{\mathbf{H}^{1,1}}\label{4.301}
\end{equation}
To estimate $\|w\|_{\mathbf{H}^{1,1}}$ we observe $\|w\|_{\mathbf{H}^{1,1}}\leq C \left(\|xw\|_{\esp^2}+\|xw_x\|_{\esp ^2}\right)$. Via Lemma \ref{Lemma hom.} we have
$$w(t)=U(t)u_0+\int _{0}^{t}U(t-\tau)|u|^2u(\tau)d\tau,$$ from \eqref{rf6} and Lemma \ref{Lemma 1.2} we note
\begin{equation}\label{4.302}
\begin{array}{cl}
 \|xw\|_{\esp^2}&\leq \displaystyle C\left(\|\mathcal{J}u_0\|_{\esp^2}+ \|u_0\|_{\mathbf X}  \right)+C\int \limits_{0}^t
 \left(\|\mathcal{J}|u|^2 u\|_{\esp^2}+ \||u|^2 u\|_{\mathbf X}  \right)d\tau\\
 &\leq \displaystyle C \varepsilon+C\int \limits_0^t \|u(\tau)\|_{\esp^\infty}^2\left(\|\mathcal{J}u\|_{\esp^2}+\|u\|_{\esp^2} \right)d\tau\\
 &\leq \displaystyle C \varepsilon+C\int \limits_0^t \varepsilon^3 \langle  \tau \rangle^{-1}\left(
  \langle  \tau \rangle^{\frac{1}{4}-\gamma}+ \langle  \tau \rangle^{\gamma}\right)d\tau\\
  &\leq C\varepsilon(1+\langle  t \rangle^{\frac{1}{4}-\gamma}+ \langle  t \rangle^{\gamma}).
\end{array} 
\end{equation}
Furthermore, from \eqref{rf60}  by a similar process like in the previous estimation we can prove
\begin{equation}\label{4.303}
\begin{array}{cl}
 \|x\partial_xw\|_{\esp^2}
  &\leq C\varepsilon(1+\langle  t \rangle^{\frac{1}{4}-\gamma}+ \langle  t \rangle^{\gamma}),
\end{array} 
\end{equation}
combining \eqref{4.302} and \eqref{4.303} we have 
$$\langle t\rangle^{-2}\|w\|_{\mathbf H^{1,1}}\leq C\varepsilon\langle t \rangle^{                                                                                                                                                                                                                                                                                                                                                                                                                                                                                                                                                                                                                                                                                                                                                                                                                                                                                                                                                                                                                                                                         -\frac{7}{4}-\gamma},$$
thus by \eqref{4.301} we have
\begin{equation}
\left\| w\left( t\right) \right\| _{\mathbf{L}^{\infty }\left( \mathbf{R}%
\right) }\leq C\varepsilon \left\langle t\right\rangle ^{-\frac{1}{2}}.
\label{4.30}
\end{equation}
This completes the proof of the lemma.
\end{proof}

By Lemma 4.1 and Lemma 4.2, we have the desired contradiction and we have a global in
time of solutions to (1.1). This completes the proof of Theorem 1.2.

\section{Asymptotic behavior of solutions \label{Section 5}}

We write 
\begin{equation}
\begin{array}{rcl}
 \psi \left( t\right)\hspace{-0.2cm}
&=&\hspace{-0.2cm}MD_t\mathcal{F}_s  \mathcal B U\left( -t\right) \psi \left( t\right) +\mathcal B^{-1}MD_t\mathcal{F}_s (M -1)\mathcal B  U\left( -t\right) \psi \left( t\right)\\
&&+(\mathcal B^{-1}-1)MD_t\mathcal{F}_s  \mathcal B U\left( -t\right) \psi \left( t\right) .
\end{array}
 \label{5.2}
\end{equation}
Hence to get the asymptotic behaviors of $\mathcal{F}_{s}\mathcal BU\left(
-t\right) u\left( t\right) $ we write 
\begin{equation}
\mathcal{F}_{s}\mathcal BU\left( -t\right) u\left( t\right) =\varphi +\psi , 
\notag
\end{equation}
where 
\begin{eqnarray*}
\varphi &=&\mathcal{F}_{s}\mathcal BU\left( -t\right) w\left( t\right) \\
\psi &=&\mathcal{F}_{s}\mathcal BU\left( -t\right) z\left( t\right) .
\end{eqnarray*}
Via \eqref{zd}, and by a direct calculation we get 
\begin{eqnarray}
\psi &=&\mathcal{F}_{s}U_{D}\left( -t\right) \partial _{x}\int_{0}^{t}e^{\frac{%
ix^{2}}{2\tau }}\frac{1}{\sqrt{\tau }}h\left( t-\tau \right) d\tau  \notag \\
&=&\mathcal{F}_{s}\partial _{x}U_{N}\left( -t\right) \int_{0}^{t}e^{\frac{%
ix^{2}}{2\tau }}\frac{1}{\sqrt{\tau }}h\left( t-\tau \right) d\tau  \notag \\
&=&i\xi \mathcal{F}_{c}U_{N}\left( -t\right) \int_{0}^{t}e^{\frac{ix^{2}}{%
2\tau }}\frac{1}{\sqrt{\tau }}h\left( t-\tau \right) d\tau  \notag \\
&=&\int_{0}^{t}i\xi e^{-\frac{i\xi ^{2}}{2}\tau }h\left( \tau \right) d\tau 
\notag \\
&=&\int_{0}^{\infty }i\xi e^{-\frac{i\xi ^{2}}{2}\tau }h\left( \tau \right)
d\tau -\int_{t}^{\infty }i\xi e^{-\frac{i\xi ^{2}}{2}\tau }h\left( \tau
\right) d\tau \notag \\
&=&A\left( \xi \right) +B\left( t,\xi \right)  
\label{4.18}
\end{eqnarray}
Also we note that the integration by parts gives 
\begin{eqnarray*}
&&B(t,\xi) =\int_{t}^{\infty }\frac{i\xi }{\left( 1-\frac{i\xi ^{2}}{2}\tau
\right) }\left( \partial _{\tau }\tau e^{-\frac{i\xi ^{2}}{2}\tau }\right)
h\left( \tau \right) d\tau \\
&=&\int_{t}^{\infty }\partial _{\tau }\left( \frac{i\xi }{\left( 1-\frac{%
i\xi ^{2}}{2}\tau \right) }\tau e^{-\frac{i\xi ^{2}}{2}\tau }h\left( \tau
\right) \right) d\tau -\int_{t}^{\infty }\tau e^{-\frac{i\xi ^{2}}{2}\tau
}\partial _{\tau }\frac{i\xi h\left( \tau \right) }{\left( 1-\frac{i\xi ^{2}%
}{2}\tau \right) }d\tau \\
&=&\frac{i\xi t}{\left( 1-\frac{i\xi ^{2}}{2}t\right) }e^{-\frac{i\xi ^{2}}{2%
}t}h\left( t\right) -\int_{t}^{\infty }\tau e^{-\frac{i\xi ^{2}}{2}\tau }%
\frac{i\xi \partial _{\tau }h\left( \tau \right) }{\left( 1-\frac{i\xi ^{2}}{%
2}\tau \right) }d\tau \\
&&-\int_{t}^{\infty }\tau e^{-\frac{i\xi ^{2}}{2}\tau }\frac{i\xi h\left(
\tau \right) \frac{i\xi ^{2}}{2}}{\left( 1-\frac{i\xi ^{2}}{2}\tau \right)
^{2}}d\tau .
\end{eqnarray*}
Hence 
\begin{eqnarray}
\left| B(t,\xi)\right| &\leq &C\sqrt{t}\left| h\left( t\right) \right|
+C\int_{t}^{\infty }\frac{\left| h\left( \tau \right) \right| }{\sqrt{\tau }}%
d\tau  +C\int_{t}^{\infty }\sqrt{\tau }\left| \partial _{\tau }h\left( \tau
\right) \right| d\tau .  \label{4.19}
\end{eqnarray}
as consequence
\begin{equation}
 |B(t,\xi)|\leq C\varepsilon \langle t \rangle^{-\frac{1}{2}-\varepsilon} \label{4.25}
\end{equation}

if $\left| h\left( t\right) \right| \leq C\varepsilon t^{-\frac{3}{4}}.$ This fact means that $B (t, \xi)$ is the remainder term.
We rewrite the nonlinear term as 
\begin{eqnarray*}
&&t^{-1}\left| \varphi +\psi \right| ^{2}\left( \varphi +\psi \right) \\
&=&t^{-1}\left| \varphi +A+B\left( t\right) \right| ^{2}\left( \varphi
+A+B\left( t\right) \right) \\
&=&t^{-1}\left| \varphi +A\right| ^{2}\left( \varphi +A\right) +R_{3}\left(
t\right) +O\left( \varepsilon ^{3}t^{-1-\varepsilon }\right) ,
\end{eqnarray*}
where 
\begin{equation*}
\left| R_{3}\left( t\right) \right| \leq Ct^{-1}\left| \varphi +A\right|
^{2}\left| B\left( t,\xi \right) \right| +Ct^{-1}\left| \varphi +A\right|
\left| B\left( t,\xi \right) \right| ^{2}
\end{equation*}
and by (\ref{4.25}) 
\begin{equation*}
\left| R_{3}\left( t\right) \right| \leq C\varepsilon \left\langle
t\right\rangle ^{-1-\varepsilon }\left| \varphi +A\right| ^{2}+C\varepsilon
^{2}\left\langle t\right\rangle ^{-1-2\varepsilon }\left| \varphi +A\right| .
\end{equation*}
By the fact that $\partial _{t}A=0$, we obtain 
\begin{equation}
i\partial _{t}\left( \varphi +A\right) =\lambda t^{-1}\left| \varphi
+A\right| ^{2}\left( \varphi +A\right) +O\left( \varepsilon
^{3}t^{-1-\varepsilon }\right) .  \label{4.26}
\end{equation}
Multiplying both sides of (\ref{4.26}) by 
\begin{equation*}
e^{i\lambda \int_{1}^{t}\tau ^{-1}\left| \varphi \left( \tau \right)
+A\right| ^{2}d\tau }
\end{equation*}
we get 
\begin{equation*}
i\partial _{t}\left( \varphi +A\right) e^{i\lambda \int_{1}^{t}\tau
^{-1}\left| \varphi \left( \tau \right) +A\right| ^{2}d\tau }=O\left(
\varepsilon ^{3}t^{-1-\varepsilon }\right) e^{i\lambda \int_{1}^{t}\tau
^{-1}\left| \varphi \left( \tau \right) +A\right| ^{2}d\tau }.
\end{equation*}
Integrating in time, we obtain 
\begin{equation*}
\left| \varphi \left( t,\xi \right) +A\left( \xi \right) \right| \leq
C\varepsilon
\end{equation*}
Hence we have 
\begin{equation*}
\left\| \varphi \left( t\right) +A\right\| _{\mathbf{L}^{\infty }\left( 
\mathbf{R}^{+}\right) }=\left\| \mathcal{F}_{s}\mathcal BU\left( -t\right) w\left(
t\right) +A\right\| _{\mathbf{L}^{\infty }\left( \mathbf{R}\right) }\leq
C\varepsilon
\end{equation*}
from which it follows again that 
\begin{equation}
\left\| \mathcal{F}_{s}\mathcal BU\left( -t\right) w\left( t\right) \right\| _{%
\mathbf{L}^{\infty } }\leq C\varepsilon .
\label{4.28}
\end{equation}
We begin with (\ref{4.26}) and we put 
\begin{equation*}
\Psi \left( t\right) =\left( \varphi +A\right) e^{i\lambda \int_{1}^{t}\tau
^{-1}\left| \varphi \left( \tau \right) +A\right| ^{2}d\tau },
\end{equation*}
then we have 
\begin{equation*}
i\partial _{t}\Psi \left( t\right) =O\left( \varepsilon
^{3}t^{-1-\varepsilon }\right)
\end{equation*}
Hence integration in time gives us 
\begin{equation}
\left\| \Psi \left( t\right) -\Psi \left( s\right) \right\| _{\mathbf{L}%
^{\infty } }\leq C\varepsilon
^{2}s^{-\varepsilon }  \label{5.5}
\end{equation}
which implies that there exists a $\Psi _{+}\in \mathbf{L}^{\infty }\left( 
\mathbf{R}\right) $ such that 
\begin{equation}
\left\| \Psi \left( t\right) -\Psi _{+}\right\| _{\mathbf{L}^{\infty }\left( 
\mathbf{R}^{+}\right) }\leq C\varepsilon ^{2}t^{-\varepsilon }.  \label{5.6}
\end{equation}
We next consider the asymptotics of the phase function. We define the
function $\Phi \left( t\right) $ as 
\begin{equation*}
\Phi \left( t\right) =\int_{1}^{t}\tau ^{-1}\left( \left| \Psi \left( \tau
\right) \right| ^{2}-\left| \Psi \left( t\right) \right| ^{2}\right) d\tau
+\int_{1}^{t}\tau ^{-1}\left( \left| \Psi \left( t\right) \right|
^{2}-\left| \Psi _{+}\right| ^{2}\right) d\tau .
\end{equation*}
Then 
\begin{equation}
\int_{1}^{t}\tau ^{-1}\left| \varphi \left( \tau \right) +A\right| ^{2}d\tau
=\Phi \left( t\right) +\left| \Psi _{+}\right| ^{2}\log t.  \label{5.6-1}
\end{equation}
By Theorem \ref{Theorem 4}, (\ref{5.5}) and (\ref{5.6}) 
\begin{eqnarray*}
&&\left\| \Phi \left( t\right) -\Phi \left( s\right) \right\| _{\mathbf{L}%
^{\infty } } \\
&=&C\int_{s}^{t}\tau ^{-1}\left\| \left| \Psi \left( \tau \right) \right|
^{2}-\left| \Psi \left( t\right) \right| ^{2}\right\| _{\mathbf{L}^{\infty
} }d\tau +C\left\| \left| \Psi \left( t\right)
\right| ^{2}-\left| \Psi _{+}\right| ^{2}\right\| _{\mathbf{L}^{\infty
} }\log \frac{t}{s} \\
&=&C\int_{s}^{t}\tau ^{-1}\left( \left\| \Psi \left( \tau \right) \right\| _{%
\mathbf{L}^{\infty } }+\left\| \Psi \left(
t\right) \right\| _{\mathbf{L}^{\infty }
}\right) \left\| \Psi \left( \tau \right) -\Psi \left( t\right) \right\| _{%
\mathbf{L}^{\infty } }d\tau \\
&&+\left( \left\| \Psi \left( t\right) \right\| _{\mathbf{L}^{\infty }\left( 
\mathbf{R}^{+}\right) }+\left\| \Psi _{+}\right\| _{\mathbf{L}^{\infty
} }\right) \left\| \Psi \left( t\right) -\Psi
_{+}\right\| _{\mathbf{L}^{\infty }\left( \mathbf{R}\right) }\log \frac{t}{s}
\\
&\leq &\int_{s}^{t}C\varepsilon ^{3}\tau ^{-1-\varepsilon }d\tau
+C\varepsilon ^{3}t^{-\varepsilon }\log \frac{t}{s},t>s
\end{eqnarray*}
which implies that there exists a real valued function $\Phi _{+}\in \mathbf{%
L}^{\infty }\left( \mathbf{R}\right) $ such that 
\begin{equation}
\left\| \Phi \left( t\right) -\Phi _{+}\right\| _{\mathbf{L}^{\infty }\left( 
\mathbf{R}\right) }\leq C\varepsilon ^{2}t^{-\varepsilon }\log t.
\label{5.7}
\end{equation}
Hence by (\ref{5.6-1}) and (\ref{5.7}) 
\begin{eqnarray}
&&\left\| \int_{1}^{t}\tau ^{-1}\left| \Psi \left( \tau \right) \right|
^{2}d\tau -\left( \left| \Psi _{+}\right| ^{2}\log t+\Phi _{+}\right)
\right\| _{\mathbf{L}^{\infty } }  \notag \\
&\leq &\left\| \Phi \left( t\right) -\Phi _{+}\right\| _{\mathbf{L}^{\infty
} }\leq C\varepsilon ^{2}t^{-\varepsilon }\log t.
\label{5.8}
\end{eqnarray}
Therefore we have by (\ref{5.6}) 
\begin{eqnarray*}
&&\left\| \Psi \left( t\right) e^{-i\lambda \int_{1}^{t}\tau ^{-1}\left|
\varphi \left( \tau \right) +A\right| ^{2}d\tau }-\Psi _{+}e^{-i\lambda
\left| \Psi _{+}\right| ^{2}\log t-i\lambda \Phi _{+}}\right\| _{\mathbf{L}%
^{\infty } } \\
&\leq &\left\| \Psi \left( t\right) -\Psi _{+}\right\| _{\mathbf{L}^{\infty
} } \\
&&+C\left\| \Psi _{+}\right\| _{\mathbf{L}^{\infty }\left( \mathbf{R}\right)
}\left\| \int_{1}^{t}\tau ^{-1}\left| \Psi \left( \tau \right) \right|
^{2}d\tau -\left( \left| \Psi _{+}\right| ^{2}\log t+\Phi _{+}\right)
\right\| _{\mathbf{L}^{\infty } } \\
&\leq &C\varepsilon ^{2}t^{-\varepsilon }\left( 1+\log t\right) .
\end{eqnarray*}
which implies 
\begin{equation*}
\left\| \left( \varphi \left( t\right) +A\right) -\Psi _{+}e^{-i\lambda
\left| \Psi _{+}\right| ^{2}\log t-i\lambda \Phi _{+}}\right\| _{\mathbf{L}%
^{\infty } }\leq C\varepsilon
^{2}t^{-\varepsilon }\left( 1+\log t\right) .
\end{equation*}
We replace $\Psi _{+}$ by $\Psi _{+}e^{i\lambda \Phi _{+}}$ to find that 
\begin{equation*}
\left\| \mathcal{F}_{s}U_{D}\left( -t\right) u\left( t\right) -\Psi
_{+}\left( \xi \right) e^{-i\lambda \left| \Psi _{+}\left( \xi \right)
\right| ^{2}\log t}\right\| _{\mathbf{L}^{\infty }\left( \mathbf{R}%
^{+}\right) }\leq C\varepsilon ^{2}t^{-\varepsilon }\left( 1+\log t\right)
\end{equation*}
from which it follows that 
\begin{equation}
\left\| u\left( t\right) -MD_{t}\Psi _{+}\left( \xi \right) e^{-i\lambda
\left| \Psi _{+}\left( \xi \right) \right| ^{2}\log t}\right\| _{\mathbf{L}%
^{\infty } }\leq C\varepsilon
^{2}t^{-\varepsilon }\left( 1+\log t\right) .  \label{5.9}
\end{equation}
This completes the proof of Theorem \ref{Theorem 6}.

\section{Proof of Theorem \protect\ref{Theorem 7}\label{Section 6}}

In order to prove Theorem \ref{Theorem 7}, we modify the function space $%
\mathbf{X}_{T}$ as follows. We define 
\begin{equation*}
\mathbf{X}_{T}=\left\{ \phi \left( t\right) \in \mathbf{C}\left( \left[ 0,T%
\right] ;\mathbf{X}\right) ;\left\| \phi \right\| _{\mathbf{X}%
_{T}}=\sup_{t\in \left[ 0,T\right] }\left\| \phi \left( t\right) \right\| _{%
\mathbf{X}}<\infty \right\} ,
\end{equation*}
where 
\begin{equation*}
\left\| \phi \left( t\right) \right\| _{\mathbf{X}}=\left(\left\langle
t\right\rangle ^{-\frac{1}{2}+\beta } \left\| \mathcal{F}_{s}\mathcal B U\left( -t\right)
\phi \left( t\right) \right\| _{\mathbf{L}^{\infty }\left( \mathbf{R}%
^{+}\right) }+\left\langle
t\right\rangle ^{-\frac{3}{2}+\beta }\left\| J\phi
\left( t\right) \right\| _{\mathbf{L}^{2}
}\right) +\left\| \phi \left( t\right) \right\| _{\mathbf{Y}},
\end{equation*}
with 
\begin{equation*}
\left\| \phi \left( t\right) \right\| _{\mathbf{Y}}=\left\langle
t\right\rangle ^{-\gamma _{1}}\varphi \left( t\right) ^{-1}\left( \left\|
\phi \left( t\right) \right\| _{\mathbf{H}^{2,0}
}+\left\| \partial _{t}\phi \left( t\right) \right\| _{\mathbf{L}^{2}\left( 
\mathbf{R}^{+}\right) }\right) ,
\end{equation*}
where 
\begin{equation*}
\varphi \left( t\right) =\left\{ 
\begin{array}{c}
\left\langle t\right\rangle ^{\frac{3}{4}-\beta },\text{ \ }\beta <\frac{3}{4%
}, \\ 
\left( \log \left\langle t\right\rangle \right) ^{2},\beta =\frac{3}{4}, \\ 
1,\frac{3}{4}<\beta \leq 1.%
\end{array}
\right. ,\gamma _{1}=\left\{ 
\begin{array}{c}
\varepsilon ^{\frac{2}{3}\left( p-1\right) },\text{ for }\beta =\frac{1}{2}+%
\frac{1}{p-1} \\ 
0,\text{ for }\frac{1}{2}+\frac{1}{p-1}<\beta \leq 1%
\end{array}
\right.
\end{equation*}
It is sufficient to prove a-priori estimates of local solutions in $\mathbf{X%
}_{T}$. We also note that $z$ is represented by the given data $h$
explicitly, therefore we need a-priori estimate of solutions $w$. We start
with the estimates 
\begin{equation*}
\left\| z\left( t\right) \right\| _{\mathbf{L}^{2}\left( \mathbf{R}%
^{+}\right) }^{2}\leq \int_{0}^{t}\left| \partial_x z\left( \tau ,0\right) \right|
\left| h\left( \tau \right) \right| d\tau ,
\end{equation*}
\begin{equation*}
\left\| \partial _{x}z\left( t\right) \right\| _{\mathbf{L}^{2}\left( 
\mathbf{R}^{+}\right) }^{2}\leq \int_{0}^{t}\left| \partial _{\tau }z\left(
\tau ,0\right) \right| \left| \partial _{x }z\left(
\tau ,0\right) \right| d\tau ,
\end{equation*}
\begin{equation*}
\left\| \partial _{t}z\left( t\right) \right\| _{\mathbf{L}^{2}\left( 
\mathbf{R}^{+}\right) }^{2}\leq\int_{0}^{t}\left| \partial _{\tau }z\left( \tau
,0\right) \right| \left| \partial _{\tau }h\left( \tau \right) \right| d\tau
.
\end{equation*}
By (\ref{3.13}) and (\ref{3.14}) with $h\left( 0\right) =0$%
\begin{eqnarray}
\left| z\left( t,0\right) \right| &\leq &C\int_{0}^{t}\frac{1}{\sqrt{\tau }}%
\left| h\left( t-\tau \right) \right| d\tau  \notag \\
&\leq &C\varepsilon ^{3}\int_{0}^{t}\frac{1}{\sqrt{\tau }\left\langle t-\tau
\right\rangle ^{\beta }}d\tau \leq C\varepsilon \left\langle
t\right\rangle ^{\frac{1}{2}-\beta }  \label{2.0-c}
\end{eqnarray}
and 
\begin{equation*}
\left| \partial _{x}z\left( t,0\right) \right|+\left| \partial _{t}z\left( t,0\right) \right| \leq C\varepsilon
^{3}\int_{0}^{t}\frac{1}{\sqrt{\tau }\left\langle t-\tau \right\rangle
^{1+\beta }}d\tau \leq C\varepsilon \left\langle t\right\rangle ^{-\frac{%
1}{2}},
\end{equation*}
if $\left| h\left( t\right) \right| +\left\langle t\right\rangle \left|
\partial _{t}h\left( t\right) \right| \leq C\varepsilon ^3\left\langle
t\right\rangle ^{-\beta }.$ Hence 
\begin{equation*}
\left| z\left( t,0\right) \right| +\left\langle t\right\rangle ^{1-\beta
}\left| \partial _{t}z\left( t,0\right) \right| \leq C\varepsilon^3
\left\langle t\right\rangle ^{\frac{1}{2}-\beta }
\end{equation*}
which implies 
\begin{eqnarray}
&&\left\| z\left( t\right) \right\| _{\mathbf{L}^{2}\left( \mathbf{R}%
^{+}\right) }^{2}+\left\| \partial _{x}z\left( t\right) \right\| _{\mathbf{L}%
^{2} }^{2}+\left\| \partial _{t}z\left( t\right)
\right\| _{\mathbf{L}^{2}}^{2}+\left\| \partial
_{x}^{2}z\left( t\right) \right\| _{\mathbf{L}^{2}\left( \mathbf{R}%
^{+}\right) }^{2}  \label{6.3} \\
&\leq &C\varepsilon ^{2}\int_{0}^{t}\left\langle \tau \right\rangle ^{\frac{1%
}{2}-2\beta }d\tau \leq C\varepsilon ^{6}\psi ^{2}\left( t\right) ,  \notag
\end{eqnarray}
where 
\begin{equation*}
\varphi \left( t\right) \geq \psi \left( t\right) =\left\{ 
\begin{array}{c}
\left\langle t\right\rangle ^{\frac{3}{4}-\beta },\text{ \ }\beta <\frac{3}{4%
}, \\ 
\log \left\langle t\right\rangle ,\beta =\frac{3}{4}, \\ 
1,\frac{3}{4}<\beta \leq 1.%
\end{array}
\right.
\end{equation*}
From \eqref{zd} we have
\begin{eqnarray}
\mathcal{F}_{s}\mathcal B U\left( -t\right) z\left( t,x\right)  &=&\mathcal{F}_{s}U_{D}\left( -t\right) z_D\left( t,x\right) \notag\\ 
&=&\int_{0}^{\infty }i\xi e^{-\frac{i\xi ^{2}}{2}\tau }h\left( \tau \right)
d\tau -\int_{t}^{\infty }i\xi e^{-\frac{i\xi ^{2}}{2}\tau }h\left( \tau
\right) d\tau .  \label{4.18 1} 
\end{eqnarray}
We note that the integration by parts gives 
\begin{eqnarray*}
&&\int_{t}^{\infty }i\xi e^{-\frac{i\xi ^{2}}{2}\tau }h\left( \tau \right)
d\tau =\int_{t}^{\infty }\frac{i\xi }{\left( 1-\frac{i\xi ^{2}}{2}\tau
\right) }\left( \partial _{\tau }\tau e^{-\frac{i\xi ^{2}}{2}\tau }\right)
h\left( \tau \right) d\tau \\
&=&\int_{t}^{\infty }\partial _{\tau }\left( \frac{i\xi }{\left( 1-\frac{%
i\xi ^{2}}{2}\tau \right) }\tau e^{-\frac{i\xi ^{2}}{2}\tau }h\left( \tau
\right) \right) d\tau -\int_{t}^{\infty }\tau e^{-\frac{i\xi ^{2}}{2}\tau
}\partial _{\tau }\frac{i\xi h\left( \tau \right) }{\left( 1-\frac{i\xi ^{2}%
}{2}\tau \right) }d\tau \\
&=&\frac{i\xi t}{\left( 1-\frac{i\xi ^{2}}{2}t\right) }e^{-\frac{i\xi ^{2}}{2%
}t}h\left( t\right) -\int_{t}^{\infty }\tau e^{-\frac{i\xi ^{2}}{2}\tau }%
\frac{i\xi \partial _{\tau }h\left( \tau \right) }{\left( 1-\frac{i\xi ^{2}}{%
2}\tau \right) }d\tau \\
&&-\int_{t}^{\infty }\tau e^{-\frac{i\xi ^{2}}{2}\tau }\frac{i\xi h\left(
\tau \right) \frac{i\xi ^{2}}{2}}{\left( 1-\frac{i\xi ^{2}}{2}\tau \right)
^{2}}d\tau .
\end{eqnarray*}
Hence 
\begin{eqnarray}
\\
\left| \int_{t}^{\infty }i\xi e^{-\frac{i\xi ^{2}}{2}\tau }h\left( \tau
\right) d\tau \right| &\leq &C\sqrt{t}\left| h\left( t\right) \right|
+C\int_{t}^{\infty }\frac{\left| h\left( \tau \right) \right| }{\sqrt{\tau }}%
d\tau  \notag +C\int_{t}^{\infty }\sqrt{\tau }\left| \partial _{\tau }h\left( \tau
\right) \right| d\tau .  \label{4.19 2} 
\end{eqnarray}
Similarly, 
\begin{eqnarray}
&&\left| \int_{0}^{\infty }i\xi e^{-\frac{i\xi ^{2}}{2}\tau }h\left( \tau
\right) d\tau \right|  \notag \\
&\leq &C\int_{0}^{\infty }\frac{\left| \tau \xi \partial _{\tau }h\left(
\tau \right) \right| }{\left| 1-\frac{i\xi ^{2}}{2}\tau \right| }d\tau
+C\int_{0}^{\infty }\frac{\left| \tau \xi ^{3}h\left( \tau \right) \right| }{%
\left| 1-\frac{i\xi ^{2}}{2}\tau \right| ^{2}}d\tau  \notag \\
&\leq &C\int_{0}^{\infty }\sqrt{\tau }\left| \partial _{\tau }h\left( \tau
\right) \right| d\tau +\int_{0}^{\infty }\frac{\left| h\left( \tau \right)
\right| }{\sqrt{\tau }}d\tau  \label{4.19-1}
\end{eqnarray}
Thus  if $\ \left| h\left( t\right) \right| \leq C\varepsilon ^{3}\left\langle
t\right\rangle ^{-\beta }$, we get by (\ref{4.18 1}) (\ref{4.19 2}) and (\ref{4.19-1}) 
\begin{equation}
\left| \mathcal{F}_{s}\mathcal B U\left( -t\right) z\right| \leq C\varepsilon ^ 3  \langle \tau \rangle^{-\beta+\frac{1}{2}}
\label{4.20}
\end{equation}
thus 
\begin{equation}
\langle \tau \rangle^{\beta-\frac{1}{2}}\left\| \mathcal{F}_{s}\mathcal B U\left( -t\right) z\right\|_{\esp ^\infty(\mathbb{R}^+)} \leq C\varepsilon^3
\end{equation}

 In the same way as in the proof of (\ref{4.10}) 
\begin{equation}
\left\| Jz\left( t\right) \right\| _{\mathbf{L}^{2}\left( \mathbf{R}%
^{+}\right) }\leq C\varepsilon^3 \langle t\rangle^{\frac{3}{2}-\beta }.  \label{6.1}
\end{equation}
Therefore via (\ref{6.3}), (\ref{4.20}) and (\ref{6.1}) 
\begin{equation}
\left\| z\left( t\right) \right\| _{\mathbf{X}}\leq C\varepsilon ^{3}
\label{6.4}
\end{equation}
for any $t>0$. To get a-priori estimates of $w$, we need the uniform time
decay of $z$. In a same way as in the proof of \eqref{4.20.0}
we have $\|z\|_{\mathbf{H}^{1,1}}\leq C\langle t \rangle^{1-\beta}$. Applying \eqref{4.0-5} and (\ref{6.4}) we get 
\begin{eqnarray}
\label{6.4-0} &&\\
\left\| z\left( t\right) \right\| _{\mathbf{L}^{\infty }\left( \mathbf{R}%
^{+}\right) } &\leq &Ct^{-\frac{1}{2}}\left\| \mathcal{F}_{s}\mathcal B U\left(
-t\right) z\left( t\right) \right\| _{\mathbf{L}^{\infty }\left( \mathbf{R}%
^{+}\right) }\notag\\
&&\hspace{1cm}+Ct^{-\frac{3}{4}}\left(\left\| Jz\left( t\right) \right\| _{\mathbf{L%
}^{2} }+\| z \|_{\mathbf L^ 2} \right)+Ct^{-2}\| z \|_{\mathbf H^{1,1}} \notag \\
&\leq &C\varepsilon ^{3}t^{-\frac{1}{2}}\left\langle t\right\rangle
^{1-\beta }+Ct^{-\frac{3}{4}}\left(\left\| Jz\left( t\right) \right\| _{\mathbf{L%
}^{2} }+\| z \|_{\mathbf L^ 2} \right)+Ct^{-2}\| z \|_{\mathbf H^{1,1}}  \notag \\
&\leq &C\varepsilon^3 t^{\frac{1}{2}-\beta }+C\varepsilon^3 t^{-\frac{3}{%
4}}\left\langle t\right\rangle ^{\frac{3}{2}-\beta }+t^{-2}\langle t\rangle^{1-\beta}\leq C\varepsilon^3 t^{%
\frac{1}{2}-\beta }  \notag 
\end{eqnarray}
for $t\geq 1$. By Sobolev, we get 
\begin{equation*}
\left\| z\left( t\right) \right\| _{\mathbf{L}^{\infty }\left( \mathbf{R}%
^{+}\right) }\leq C\left\| z\left( t\right) \right\| _{\mathbf{L}^{2}\left( 
\mathbf{R}^{+}\right) }^{\frac{1}{2}}\left\| \partial _{x}z\left( t\right)
\right\| _{\mathbf{L}^{2}}^{\frac{1}{2}}\leq
C\varepsilon^3
\end{equation*}
for $t\leq 1.$ Therefore 
\begin{equation}
\left\| z\left( t\right) \right\| _{\mathbf{L}^{\infty }\left( \mathbf{R}%
^{+}\right) }\leq C\varepsilon^3\left\langle t\right\rangle ^{\frac{1}{2}%
-\beta }  \label{6.2}
\end{equation}
for $t\geq 0.$

By Theorem \ref{Theorem 2}, we may assume that 
$\left\| w\right\| _{\mathbf{X}_{1}}\leq 3\varepsilon$
if we take $\varepsilon $ small enough. We now prove that for any time $%
\widetilde{T},$ the estimate 
$\left\| w\right\| _{\mathbf{X}_{\widetilde{T}}}^{2}<\varepsilon ^{\frac{4}{3}}$
holds. If the above estimate does not hold, then we can find a finite time $T $ such that 
$\left\| w\right\| _{\mathbf{X}_{T}}^{2}=\varepsilon ^{\frac{4}{3}}.$
In the below we show that $T$ satisfying the above identity does not hold.
This is the desired contradiction.

Applying the energy method to the equation of $w$ and from the same argument as in our previous paper we  obtain 

\begin{eqnarray}
&&\left\| w\left( t\right) \right\| _{\mathbf{L}^{2}\left( \mathbf{R}%
^{+}\right) }+\left\| \partial _{x}w\left( t\right) \right\| _{\mathbf{L}%
^{2} }  \notag \\
&&+\left\| \partial _{t}w\left( t\right) \right\| _{\mathbf{L}^{2}\left( 
\mathbf{R}^{+}\right) }+\left\| \partial _{x}^{2}w\left( t\right) \right\| _{%
\mathbf{L}^{2} }\leq C\varepsilon \varphi \left(
t\right) .  \label{6.6}
\end{eqnarray}

\begin{equation}
\left\| Jw\left( t\right) \right\| _{\mathbf{L}^{2}\left( \mathbf{R}%
^{+}\right) }\leq C\varepsilon \left\langle t\right\rangle ^{\frac{5}{4}%
-\beta }  \label{6.7}
\end{equation}
for $\beta =\frac{1}{2}+\frac{1}{p-1}$ and 
\begin{equation}
\left\| Jw\left( t\right) \right\| _{\mathbf{L}^{2}\left( \mathbf{R}%
^{+}\right) }\leq C\varepsilon \left\langle t\right\rangle ^{\left( \frac{5}{%
4}-\beta \right) -\delta }  \label{6.8}
\end{equation}
for $\beta >\frac{1}{2}+\frac{1}{p-1}.$ In order to get the a-priori
estimate of $\left\| \mathcal{F}_{s}\mathcal BU\left( -t\right) w\right\| _{\mathbf{L}%
^{\infty } }$, we start split the solution $u=y+e^{-x}h$. By the factorization formula
of $U\left( -t\right) $, we get in the same way as in the proof of (\ref{H}) 
\begin{equation}\label{6.9}
   \begin{array}{lcr}
  i\partial _{t}\mathcal{F}_s \mathcal B  U(-t) \left( y+e^{-x}h\left(
t\right) \right)\\
=t^{-\frac{p-1}{2}}\lambda\mathcal{F}_s  M^{-1}\F_s ^{-1}\left|\F_sM \mathcal B U\left( -t\right) u \right|^ 2  \F_sM \mathcal B U\left( -t\right) u\\
\hspace{7cm}+R_2+R_3+G.
\end{array}
 \end{equation}
 where
$$G=i\partial _{t}\mathcal{F}_s \mathcal B U(-t)  z$$
$$
 R_2=3\lambda t^{-\frac{p-1}{2}}\mathcal{F}_s ( M^{-1}-1)\F_s^{-1} \left|\F_sM \mathcal B U\left( -t\right) y \right|^ 2  \F_sM \mathcal B U\left( -t\right) y
$$
$$R_3=3\lambda t^{-\frac{p-1}{2}}\left(\left|\F_sM \mathcal B U\left( -t\right) y \right|^ 2  \F_sM \mathcal B U\left( -t\right) y-\left|\F_s \mathcal B U\left( -t\right) y \right|^ 2  \F_s\mathcal B U\left( -t\right) y\right)$$
By the similar calculations as in (\ref{C}) and (\ref{D}), we have 
\begin{eqnarray*}
&&\left\| R_{2}\left( t\right) \right\| _{\mathbf{L}^{\infty }(\mathbf{R}^{+})}+ \left\| R_{3}\left( t\right) \right\| _{\mathbf{L}^{\infty }(\mathbf{R}^{+})}\\
&&\leq C \epsilon^{\frac{2}{3}p}t^{-\frac{p-1}{2}-\frac{1}{4}}\| \mathcal F_s\mathcal BU(-t)u\|_{\esp^\infty}^{p-1}\|Ju\|_{\esp^2}+Ct^{-\frac{p-1}{2}-\frac{1}{4}p} \|Ju\|^{p}_{\esp^2}\\
&&\leq C\varepsilon ^{\frac{2}{3}p}t^{-\frac{p-1}{2}-\frac{1}{4}}t^{\left(
1-\beta \right) \left( p-1\right) }t^{\frac{5}{4}-\beta }+C\varepsilon ^{%
\frac{2}{3}p}t^{-\frac{p-1}{2}}t^{\left( 1-\beta \right) p}
\end{eqnarray*}
If $\beta =\frac{1}{2}+\frac{1}{p-1},$ then 
\begin{equation*}
\left\| R_{1}\left( t\right) \right\| _{\mathbf{L}^{\infty }(\mathbf{R}%
^{+})}+\left\| R_{2}\left( t\right) \right\| _{\mathbf{L}^{\infty }(\mathbf{R%
}^{+})}\leq C\varepsilon ^{\frac{2}{3}p}t^{-\beta }.
\end{equation*}
If $\beta >\frac{1}{2}+\frac{1}{p-1},$ then 
\begin{equation*}
\left\| R_{1}\left( t\right) \right\| _{\mathbf{L}^{\infty }(\mathbf{R}%
^{+})}+\left\| R_{2}\left( t\right) \right\| _{\mathbf{L}^{\infty }(\mathbf{R%
}^{+})}\leq C\varepsilon ^{\frac{2}{3}p}t^{-(\frac{1}{2}+\frac{1}{p+1}%
)-\delta p},
\end{equation*}
where $\delta =\beta -\left( \frac{1}{2}+\frac{1}{p-1}\right) >0.$

\ By (\ref{6.9}) we have 
\begin{eqnarray*}
&&i\partial _{t}\mathcal{F}_{c}U\left( -t\right) (y+e^x h(t)) \\
&=&\lambda t^{-\frac{p-1}{2}}\left| \mathcal{F}_{c}U\left( -t\right)
u\right| ^{p-1}\mathcal{F}_{c}U\left( -t\right) u+O\left( \varepsilon ^{%
\frac{2}{3}p}t^{-\beta }\right)+G \\
&=&O\left( \varepsilon ^{\frac{2}{3}p}t^{\frac{p+1}{2}}t^{-\beta p}\right)
+O\left( \varepsilon ^{\frac{2}{3}p}t^{-\beta }\right) +G\\
&=&O\left( \varepsilon ^{\frac{2}{3}p}t^{-\beta }\right)+G
\end{eqnarray*}
for $\beta =\frac{1}{2}+\frac{1}{p-1}$ and 
\begin{equation*}
i\partial _{t}\mathcal{F}_{c}U\left( -t\right) w=O\left( \varepsilon ^{\frac{%
2}{3}p}t^{-\beta -\delta p}\right)+G
\end{equation*}
for $\beta >\frac{1}{2}+\frac{1}{p-1}.$ From the previous estimates and via a similar process as in the proof of \eqref{4.22}, we can prove
\begin{equation}
\left\| \mathcal{F}_{c}U\left( -t\right) w\right\| _{\mathbf{L}^{\infty
} }\leq \varepsilon +C\varepsilon ^{\frac{2}{3}%
p}\left\langle t\right\rangle ^{1-\beta }\leq C\varepsilon \left\langle
t\right\rangle ^{1-\beta }.  \label{7.1}
\end{equation}
By (\ref{6.6}), (\ref{6.7}), (\ref{6.8}) and (\ref{7.1}) we get 
\begin{equation*}
\left\| w\right\| _{\mathbf{X}_{T}}\leq C\varepsilon <\varepsilon ^{\frac{2}{%
3}}.
\end{equation*}
This is the desired contradiction. This completes the proof of Theorem \ref%
{Theorem 7}.

\section{Proof of Theorem \protect\ref{Theorem 8}\label{Section 7}}

In order to prove Theorem \ref{Theorem 8}, it is sufficient to consider the
asymptotic behavior of $z$ since by Theorem \ref{Theorem 7}, we know that 
\begin{equation*}
\left\Vert w\left( t\right) \right\Vert _{\mathbf{L}^{\infty }\left( \mathbf{%
R}^{+}\right) }\leq C\varepsilon t^{\frac{1}{2}-\beta -\delta },
\end{equation*}%
where $\delta =\beta -\frac{p+1}{2(p-1)}\geq 0$ \ and 
\begin{equation*}
u\left( t,x\right) =z\left( t,s\right) +w\left( t,x\right) .
\end{equation*}%
We have from \eqref{bfz}
$$z(t,x)=\frac{1}{2\pi i}\int \limits_0^ t h(\tau) I(t-\tau,x)\ d\tau, \ \text{with \ } I(s,x)=\int \limits_{\mathbb{R}}e^{ipx}e^{ip^2 s}\frac{p}{1+i\alpha p} dp.$$
Via the stationaty phase method we have
$$I(s,x)=\sqrt{\frac{i}{\pi s}}e^{i\frac{x^2}{2s}}\frac{x/2\sqrt{s}}{1-i\alpha x/2\sqrt{s}}+I_1(s,x), \ \text{where} \ I_1(s,x)=O(\langle s\rangle ^{-\frac{1}{2}-\frac{1}{4}})$$
thus 
\begin{equation}
z(t,x)=z_1(t,x)+z_2(t,x) ,\label{01}
\end{equation}
where
\begin{equation}
\begin{array}{c}
z_1(t,x)=\displaystyle \frac{1}{2\sqrt{\pi^3 i}}\int \limits_0^ t h(\tau)\frac{ e^{i\frac{x^2}{2(t-\tau)}}}{\sqrt{t-\tau}}\frac{x/2\sqrt{t-\tau}}{1-i\alpha x/2\sqrt{t-\tau}}\ d\tau,\ \ 
z_2(t,x)\frac{1}{2\pi i}\int \limits_0^ t h(\tau) I_1(t-\tau,x)\ d\tau.
\end{array}
\end{equation}
On the other hand, we have 
\begin{equation*}
h\left( t\right) =A\frac{t}{(1+t)^{1+\beta }}+h_{1}\left( t\right) ,
\end{equation*}%
where $\left\vert \partial _{t}h\left( t\right) \right\vert \leq
\left\langle t\right\rangle ^{-1-\beta }$ and $h_{1}\left( t\right)
=O(t\left\langle t\right\rangle ^{-2-\gamma })$ for $\gamma >0$. As
consequence  it follows 
\begin{equation}
z_1\left( t,x\right) =z_{11}(t,x)+z_{12}(t,x),  \label{001}
\end{equation}%
where 
\begin{equation*}
z_{11}(t,x)=\frac{A}{i\sqrt{2i\pi }}\int_{0}^{t}\frac{\tau }{(1+\tau )^{1+\beta }}e^{i\frac{x^2}{2(t-\tau)}}\frac{x/2\sqrt{t-\tau}}{1-i\alpha x/2\sqrt{t-\tau}}
d\tau ,
\end{equation*}%
\begin{equation*}
z_{12}(t,x)=\frac{A}{i\sqrt{2i\pi }}\int_{0}^{t}h_{1}(\tau )I_1(t-\tau,x)d\tau .
\end{equation*}
By a changing of
variable $\frac{\tau }{t}=y$, we get 
\begin{equation*}
z_{11}\left( t,x\right) =t^{\frac{1}{2}-\beta }\frac{A}{i\sqrt{2i\pi }}%
\int_{0}^{1}e^{\frac{i\xi ^{2}}{2(1-y)}}\frac{1}{\sqrt{1-y}}\frac{y}{(y+%
\frac{1}{t})^{1+\beta }}\frac{\xi}{2\sqrt{1-y}-i\alpha \xi}dy
\end{equation*}%
for $\xi =\frac{x}{\sqrt{t}}.$ \ Applying 
\begin{eqnarray*}
&&\left\vert \frac{1}{y^{1+\beta }}-\frac{1}{(y+\frac{1}{t})^{1+\beta }}%
\right\vert  \\
&\leq &C\left\vert \frac{1}{y^{1+\beta }}\right\vert \left\vert 1-\frac{1}{%
\left( 1+\frac{1}{ty}\right) ^{1+\beta }}\right\vert \leq C\left\vert \frac{1%
}{y^{1+\beta }}\right\vert \left\vert \frac{1}{\left( 1+\frac{1}{ty}\right)
^{2+\beta }}\right\vert \left\vert \frac{1}{ty}\right\vert  \\
&\leq &C\left\vert \frac{1}{y^{1+\beta }}\right\vert \left\vert \frac{1}{y}%
\right\vert ^{\gamma }t^{-\gamma },0\leq \gamma \leq 1,
\end{eqnarray*}%
we get 
\begin{equation}
\begin{array}{l}
z_{11}\left( t,x\right) =\displaystyle t^{\frac{1}{2}-\beta }\frac{A}{i\sqrt{2i\pi }}%
\int_{0}^{1}e^{\frac{i\xi ^{2}}{2(1-y)}}\frac{1}{y^{\beta }\sqrt{1-y}}\frac{\xi}{2\sqrt{1-y}-i\alpha \xi}dy\\
\hspace{2cm}+\displaystyle t^{%
\frac{1}{2}-\beta -\gamma _{1}}O\left( \int_{0}^{1}\frac{1}{y^{\beta +\gamma
_{1}}\sqrt{1-y}}dy\right) 
\end{array} \label{0011}
\end{equation}%
for $0<\gamma _{1}<1-\beta $. In a similar way we can prove 
\begin{equation}
\left\Vert z_{12}(t)\right\Vert _{\esp^{\infty }}\leq C\left\langle
t\right\rangle ^{-\frac{1}{2}}  \label{002}
\end{equation}
\begin{equation}
\left\Vert z_{2}(t)\right\Vert _{\esp ^{\infty }}\leq C\left\langle
t\right\rangle ^{-\frac{1}{2}}\label{00}
\end{equation}
since $h_{1}\left( t\right) =O(t\left\langle t\right\rangle ^{-2-\gamma }).$
\ This estimate implies that $z_{12}$ and $z_2$ are the remainder terms.  As consequence from (\ref{001}) and (\ref%
{002}) it follows that  
\begin{eqnarray}
z\left( t,x\right)  &=&t^{\frac{1}{2}-\beta }A\frac{1}{i\sqrt{2i\pi }}%
\int_{0}^{1}e^{\frac{i\xi ^{2}}{2(1-y)}}\frac{1}{y^{\beta }\sqrt{1-y}}%
dy+O(t^{\frac{1}{2}-\beta -\gamma _{1}})+O(t^{-\frac{1}{2}})  \notag \\
&=&t^{\frac{1}{2}-\beta }A\frac{1}{i\sqrt{2i\pi }}\int_{0}^{1}e^{\frac{i\xi
^{2}}{2(1-y)}}\frac{1}{y^{\beta }\sqrt{1-y}}dy+O(t^{\frac{1}{2}-\beta
-\gamma _{1}})  \label{010}
\end{eqnarray}%
for $\xi =xt^{-\frac{1}{2}},0<\gamma _{1}<1-\beta .$ \ Finally from (\ref%
{010}) we get 
\begin{eqnarray*}
u(t,x) &=&z(t,x)+w(x,t) \\
&=&t^{\frac{1}{2}-\beta }A\Lambda (xt^{-\frac{1}{2}})+O(t^{\frac{1}{2}-\beta
-\gamma _{1}})+O(t^{\frac{1}{2}-\beta -\delta }) \\
&=&t^{\frac{1}{2}-\beta }A\Lambda (xt^{-\frac{1}{2}})+O(t^{\frac{1}{2}-\beta
-\delta }),
\end{eqnarray*}%
where $\gamma _{1}<1-\beta ,$\ $\delta =\beta -\frac{p+2}{2(p-1)}$ and 
\begin{equation*}
\Lambda (\xi )=\frac{1}{i\sqrt{2i\pi }}\int_{0}^{1}e^{\frac{i\xi ^{2}}{2(1-y)%
}}\frac{1}{y^{\beta }\sqrt{1-y}}\frac{\xi}{2(1-y)-i\alpha \xi}dy.
\end{equation*}%
This completes the proof of the theorem.

\textbf{Acknowledgments. }The work of N.H. is partially supported by JSPS
KAKENHI Grant Numbers JP25220702, JP15H03630. The work of E.I.K. is
partially supported by CONACYT 252053-F and PAPIIT project IN100817.  The work of L.E is supported by GSSI.

\end{document}